\documentclass[11pt]{article}
\usepackage{verbatim} 
\usepackage{graphicx}
\usepackage[utf8]{inputenc}
\usepackage{amsmath,amsfonts,amssymb,amscd,amsthm,txfonts,hyperref,float}
\usepackage{xcolor}
\newtheorem{theorem}{Theorem}[section]
\newtheorem{proposition}[theorem]{Proposition}
\newtheorem{lemma}[theorem]{Lemma}
\newtheorem{corollary}[theorem]{Corollary}
\newtheorem{remark}{Remark}[section]

\theoremstyle{definition}
\newtheorem{definition}[theorem]{Definition}

\newcommand{\R}{\mathbb{R}}
\newcommand{\C}{\mathbb{C}}
\newcommand{\N}{\mathbb{N}}
\newcommand{\Z}{\mathbb{Z}}
\newcommand{\T}{\mathbb{T}}

\newcommand{\lap}{\bigtriangleup}

\newcommand{\E}{\mathbb E}
\renewcommand{\Re}{\mbox{Re}}
\newcommand{\D}{\varmathbb D}
\newcommand{\re}{\textrm{Re}}

\newcommand{\an}[1]{\langle #1 \rangle}

\begin{document}

\title{Singularities in the weak turbulence regime for the quintic Schrödinger equation}
\author{Anne-Sophie de Suzzoni\footnote{CMLS, \'Ecole Polytechnique, CNRS, Universit\'e Paris-Saclay, 91128 PALAISEAU Cedex, France, \texttt{anne-sophie.de-suzzoni@polytechnique.edu}}}

\maketitle

\begin{abstract} In this paper, we discuss the problem of derivation of kinetic equations from the theory of weak turbulence for the quintic Schrödinger equation. We study the quintic Schrödinger equation on $L\T$, with $L\gg 1$ and with a non-linearity of size $\varepsilon\ll 1$. We consider the correlations $f(T)$ of the Fourier coefficients of the solution at times $t = T\varepsilon^{-2}$ when $\varepsilon\rightarrow 0$ and $L\rightarrow \infty$. Our results can be summed up in the following way : there exists a regime for $\varepsilon$ and $L$ such that for $T$ dyadic, $f(T)$ has the form expected from the physics literature, but such that $f$ has an infinite number of discontinuity points.
\end{abstract}

\tableofcontents

\section{Introduction}

In this paper, we discuss the problem of derivation of kinetic equations from the theory of weak turbulence for the quintic Schrödinger equation. 

Weak turbulence has been introduced by Zakharov, see for example \cite{zakharov1967}, and an extensive literature has been developed in Physics, as it is reviewed in the book by Nazarenko, \cite{Naz}. 

We first describe the expected result before describing our framework and results.

Consider the quintic Schrödinger Cauchy problem :
\begin{equation}\label{quinticCauchyprob}
\left \lbrace{\begin{array}{c}
i\partial_t u_{\varepsilon,L} = -\lap u_{\varepsilon,L} + \varepsilon |u_{\varepsilon,L}|^4 u_{\varepsilon,L} \\
u_{\varepsilon,L}(t=0) = a_L
\end{array}} \right.
\end{equation}
on the torus $L\T$, $L\gg 1$, with $\varepsilon \ll 1$ and with initial datum
\[
a_L = \sum_{k\in \Z} a(k/L)\frac{e^{ikx/L}}{\sqrt{2\pi L}} g_k
\]
where $(g_k)_k$ is a sequence of independent centered and normalized Gaussian variables, and $a$ is a smooth, compactly supported map. 

The map $u_{\varepsilon,L}$ is a map from $\R\times (L\T)\times \Omega$ where $\Omega$ is a probability space supporting the $(g_k)_k$, and $\lap$ is the Laplace-Beltrami operator on $L\T$. 

The issue at stake is the description of the dynamics of
\[
\E(|\hat u_{\varepsilon,L}(t,k)|^2)
\]
as $\varepsilon$ goes to $0$ and $L$ goes to $\infty$, for any $k\in \frac1{L}\Z$ , where 
\[
\hat u_{\varepsilon,L}(t,k) := \frac1{\sqrt{2\pi L}} \int_{L\T} e^{-ikx} u_{\varepsilon,L}(t,x)dx
\] 
is the Fourier transform of $u_{\varepsilon,L}$.

Note that because the law of the initial datum $a_L$ and the equation are invariant under the action of space translations, we have 
\[
\E(\overline{\hat u_{\varepsilon,L}(t,k')} \hat u_{\varepsilon,L}(t,k)) = 0
\]
at all times $t\in \R$, if $k\neq k'$. 

The ersatz is the following. Approaching $u_{\varepsilon,L}$ by its $0$-th, first and second Picard iterates, we get up to second order in $\varepsilon$,
\[
u_{\varepsilon,L}(t) \sim e^{it\lap }a_L + \varepsilon b_L(t) + \varepsilon^2 c_L(t)
\]
where $b_L$ and $c_L$ are solutions to the equations
\[
\left \lbrace{\begin{array}{c}
i\partial_t b_L = -\lap b_L + |e^{it\lap }a_L|^4(e^{it\lap }a_L) \\
i\partial_t c_L = -\lap c_L + 3 |e^{it\lap }a_L|^4 b_L + 2\bar{b_L} |e^{it\lap }a_L|^2(e^{it\lap }a_L)^2\end{array}}\right.
\]
with initial datum $b_L(t=0) = c_L(t=0) = 0$.

Up to second order, we thus get
\[
\partial_t \E(|\hat u_{\varepsilon,L}(t,k)|^2) \sim 2\varepsilon \Re \Big( \overline{\widehat{e^{it\lap}a_L}(t,k)}\partial_t \hat b_L (t,k) \Big) + 2 \varepsilon^2 \Re \Big( \overline{\widehat{e^{it\lap}a_L}(t,k)}\partial_t \hat c_L (t,k) + \overline{\hat b_L (t,k)}\partial_t \hat b_L (t,k)\Big)
\]
Because of probabilistic cancellations due to the law of the initial datum $a_L$, the term
\[
\Re \Big( \overline{\widehat{e^{it\lap}a_L}(t,k)}\partial_t \hat b_L (t,k) \Big)
\]
involves only first order trivial resonances, which can be removed from the solution by multiplicating y a phase, see \cite{denghani19,LukSpo} for a discussion in different contexts. 

This suggests that the right scale of time is $\varepsilon^{-2}$, and thus we consider the quantity
\[
U_{\varepsilon,L}(t,k) = \E(|\hat u_{\varepsilon,L}(t\varepsilon^{-2},k)|^2)
\]
with
\[
\partial_t U_{\varepsilon,L} (t,k) \sim 2\Re \Big( \overline{\widehat{e^{it\varepsilon^{-2}\lap}a_L}(k)}\partial_t \hat c_L (t\varepsilon^{-2},k) + \overline{\hat b_L (t\varepsilon^{-2},k)}\partial_t \hat b_L (t\varepsilon^{-2},k)\Big)
\]
This is valid only if the right hand side is not null.

By taking $L\rightarrow \infty$ and $\varepsilon\rightarrow 0$, we expect to get
\begin{equation}\label{expectedresult}
U_{\varepsilon,L}(t,k) = \frac3{2\pi^3}\int_{\R^5} \delta(k + \sum_{j=1}^5 (-1)^j k_j)  \delta (\Delta (k,\vec k) ) \frac1{k-k_1 + k_2 - k_3 }|a(k)|^2 \prod_{j=1}^5|a(k_j)|^2F_a(k,\vec k)d\vec k
\end{equation}
where $\vec k = (k_1,\hdots,k_5)$, where
\[
F_a(k,\vec k) = \Big( \frac1{|a(k)|^2}-\frac1{|a(k_1)|^2}+\frac1{|a(k_2)|^2}-\frac1{|a(k_3)|^2}+\frac1{|a(k_4)|^2}-\frac1{|a(k_5)|^2}\Big),
\]
where $\Delta (k,\vec k) = k^2 - k_1^2+k_2^2-k_3^2+k_4^2 - k_5^2$ and where the $\delta$s are Dirac deltas. The integral converges.

The aim of this paper is to illustrate the fact that this limit is far from being obvious. The reason is that the Picard expansion does not uniformly converge in $L$ for times of order $\varepsilon^{-2}$, the $L^2$ Lebesgue norm of the initial datum $a_L$ growing with $L$. The series of paper \cite{BGHS,CoG19,CoGe20,denghani19} have reached larger and larger times by developing fine analytic estimates for the cubic Schrödinger equation due to the algebraic structure of the non-linearity. In particular, they estimate so-called Dyson series. In the series of papers \cite{DyKuk3,DyKuk2,DyKuk1}, Dymov and Kuksin reach a satisfying result at times of order $\varepsilon^{-1}$, by using quasi-solutions for the cubic Schrödinger equation, a notion coming from the Physics literature. We take here an approach closer to their point of view. We also mention the works \cite{LukSpo} on cubic equations on nets, and \cite{dS15,dSTont} on quadratic equations coming fluid mechanics. For a general bibliographical review on connected subjects, we refer to \cite{CoGe20} and references therein. 

Apart from the model (quintic Vs cubic or quadratic equations), chosen to get a nontrivial manifold for the first order resonances, even when the dimension is 1, the main differences with the above literature are threefold. One important aspect of this work is that, to justify  
the use of quasi-solutions, we use the stochastic tool of Wick products. A side advantage of this is that we avoid certain possible correlations, that would diverge at times of order $\varepsilon^{-2}$. Another main difference is that, in order to reach times of order $\varepsilon^{-2}$, we prepare the approximation of the equation in $L\T$. We justify this approximation in the next subsections. It allows to reach very large times. Finally, we exhibit a regime for $\varepsilon$ and $L$ such that depending on the time $t\in \R$, we can get different behaviors for the sequence $U_{\varepsilon,L}(t,k)$. 

We now describe our framework and results.

\subsection{Framework}

Let $(\xi_k)_{k\in \Z}$ be a sequence of real, independent, centered and normalized Gaussian variables. We write $(\Omega, \mathcal F, \mathbb P)$ their underlying probability space and $\mathcal A$ the $\sigma$-algebra generated by the $(\xi_k)_k$. Given a well-chosen sequence $(q_k)_{k\in \Z}$ of positive real numbers, we define $\mathcal S_{-1}(H^s(L\T))$, $s>\frac1{2}$ the space of Kondratiev's distributions of $(\Omega, \mathcal A, \mathbb P)$ on $H^s(L\T)$. For the exact definition of Kondratiev's distributions, we refer to Subsection \ref{subsec:Kondratiev}. Here, we use the terminology of \cite{KL2010,LOP2004}, and for the general definition of Kondratiev's distributions, we refer to the original work \cite{KSWY1998}. The Wick product (see again Subsection \ref{subsec:Kondratiev}) is well-defined on $\mathcal S_{-1}(H^s(L\T))$, we denote it $\diamond$. 

Let $F_L$ be the map defined on Fourier mode by, for all $k\in \frac1{L}\Z$,
\[
\widehat{F_L} (\alpha,\beta,\gamma,\delta,\eta)(k) = \frac1{(2\pi L)} \sum_{C_L(k)} \hat \alpha (k_1) \diamond \overline{\hat \beta (k_2)} \diamond \gamma(k_3) \diamond \overline{\delta}(k_4) \diamond \eta(k_5)
\]
where
\begin{multline}\label{defCL}
C_L(k) = \Big\{ (k_1,k_2,k_3,k_4,k_5) \in \Big( \frac1{L}\Z\Big)^5 \;\Big|\; k_1-k_2+k_3-k_4+k_5 = k,  \\ 
 |k-k_1+k_2-k_3|\geq \mu^{-1}(L), \quad
|k^2 -k_1^2 + k_2^2 - k_3^2+k_4^2-k_5^2| \geq \nu^{-1}(L) \Big\}
\end{multline}
and where we used the abuse of notation
\[
\hat \alpha(k_1) \diamond \hat \beta (k_2) = \Big( (\hat \alpha(k_1) e^{ik_1x}) \diamond (\hat \beta (k_2) e^{ik_2 x}) \Big) e^{-i(k_1+k_2)x}.
\]

Note that taking $(u_L)$ a sequence of maps such that $u_L \in \mathcal S^{-1}(H^s(L\T))$ and such that for any $\alpha \in \N_f^\Z$, the sequence $\|(u_L)_\alpha\|_{H^s(L\T)}$ is uniformly bounded in $L$ (see Remark \ref{rem:boundSol} for the relevance of this property), then we have that for all smooth and compactly supported map $f$, and all $\alpha \in \N_f^\Z$,
\[
\an{f,F_L(u,u,u,u,u)_\alpha - (u\diamond \bar u\diamond u \diamond \bar u\diamond u)_\alpha} \underset{L\rightarrow \infty}{\rightarrow} 0
\]
where $\an{\cdot,\cdot}$ is the inner product in $\R$, as soon as $\nu$ goes to $\infty$ with $L$. 

We consider the Cauchy problem
\begin{equation}\label{improvedCauchyprob}
\left \lbrace{\begin{array}{c}
i\partial_t u_{\varepsilon,L} = -\lap u_{\varepsilon,L} + \varepsilon F_L(u,u,u,u,u)\\
u_{\varepsilon,L}(t=0) = a_L
\end{array}} \right.
\end{equation}
where $a_L$ is given by
\[
a_L = \sum_{k\in \Z} a(k/L) \frac{e^{ikx/L}}{\sqrt{2\pi L}} g_k
\]
where $g_k = \frac1{\sqrt 2} (\xi_{\varphi(0,k)}+i\xi_{\varphi(1,k)})$ with $\varphi$ a bijection from $\{0\}\times \Z\sqcup \{1\}\times \Z$ to $\Z$. The sequence $(g_k)_k$ is a sequence of independent, centered, normalized complex Gaussian variables. This is the classical way of extending real Gaussian Hilbert spaces to complex ones. 

We call $P_N$ the projection onto the Wiener chaos of degree at most $N$.

We are now ready to state the results. 

\subsection{Results} 

\begin{theorem}\label{theo:main1} For any $L \in \N^*$, there exists a Banach algebra $X \subseteq \mathcal S_{-1}(H^s(L \T))$ into which the Cauchy problem \eqref{improvedCauchyprob} is globally well-posed and such that $a_L \in X$. Set $N \in \N^*$, $M\geq N$, and set $f$ and $g$ two smooth compactly supported maps of $\R$. 

Assume that $\varepsilon $ writes
\[
\varepsilon^{-2} = 2\pi L^2 2^L + \rho(L),
\]
that $\rho, \nu$ and $\mu$ satisfy the following relationships : 
\begin{equation}\label{assumptionsRegime}
\left \lbrace{\begin{array}{c}
\exists \alpha >0,\, \nu(L)^{1+\alpha} = o(L^{1/2}),\\
\rho(L) = o(\mu(L)),\\ 
\rho(L)\mu(L) = o(\nu(L)),\\ 
\ln^2(\mu(L)) = o(\rho^{1/4}),\\ 
\mu(L) \rightarrow \infty,\\ 
\rho(L) \rightarrow \infty
\end{array}} \right.
\end{equation} 
and finally assume that $t$ is dyadic. 

Then, we have 
\begin{multline*}
\lim_{L \rightarrow \infty} \partial_t \E(\an{P_N U_L,f} \an{g,P_M U_L } ) (t)=\\
 \frac3{4\pi^4}\int_{\R^6} \delta(k + \sum_{j=1}^5 (-1)^j k_j)  \delta (\Delta (\vec k) ) \frac1{k-k_1 + k_2 - k_3 } \prod_{j=1}^5|a(k_j)|^2 \hat f(k) \overline{\hat g (k)} dk \prod_{j=1}^5 dk_j
\end{multline*}
and besides the integral converges.

Above, we have $U_L = u_{\varepsilon,L}(t\varepsilon^{-2})$, $\Delta(\vec k) = k^2 - k_1^2 + k_2^2 - k_3^2 + k_4^2 - k_5^2$ and the $\delta$s are Dirac deltas. 
\end{theorem}

\begin{remark} The regime $\nu = L^\alpha$, $\mu = L^\beta$, $\rho = L^\gamma$ with $0<\gamma<\beta<\alpha < \frac12$ and $\beta+\gamma <\alpha$ satisfies the assumptions \eqref{assumptionsRegime}.
\end{remark}

\begin{remark} On the one hand, this theorem applies to a class of times that are dense in $\R$. If the dyadic numbers are not satisfying, they can be changed to rational numbers by choosing the regime
\[
\varepsilon^{-2} = 2\pi L^2 L! + \rho(L).
\]
However, this will serve us to prove that 
\[
t\mapsto \lim_{L \rightarrow \infty} \partial_t \E(\an{P_N U_L,f} \an{g,P_M U_L } ) (t)
\]
has a chaotic behavior. 
\end{remark}

\begin{theorem}\label{theo:main2} With the same notations as in Theorem \ref{theo:main1}, assuming that $\varepsilon$ writes 
\[
\varepsilon^{-2}= 2\pi L^2 2^L + \rho(L),
\]
that $\rho,\nu$ and $\mu$ satisfy \eqref{assumptionsRegime}  and that $t \in \frac13 + \D$, $\D$ being the set of dyadic numbers, we have 
\begin{multline*}
\lim_{L \rightarrow \infty} \partial_t \E(\an{P_N U_L,f} \an{g,P_M U_L } ) (t)= \\ \frac1{12\pi^4}\int_{\R^6} \delta(k + \sum_{j=1}^5 (-1)^j k_j)  \delta (\Delta (\vec k) ) \frac1{k-k_1 + k_2 - k_3 } \prod_{j=1}^5|a(k_j)|^2 \hat f(k) \overline{\hat g (k)} dk \prod_{j=1}^5 dk_j
\end{multline*}
and besides the integral converges.
\end{theorem}

\begin{remark} The difference is in the constant in front of the integral.\end{remark}

\begin{remark} The behavior on the sequence  $\partial_t \E(\an{P_N U_L,f} \an{g,P_M U_L } )$ depends mainly on the behavior of the sequence
\[
2^L t - \lfloor 2^L t\rfloor.
\]
In the case when $t$ is rational, because the sequence $2^L t -  \lfloor 2^L t\rfloor$ is pre-periodic, we believe that 
\begin{multline*}
\lim_{L \rightarrow \infty} \partial_t \E(\an{P_N U_L,f} \an{g,P_M U_L } ) (t)= \\
C(t)\int_{\R^6} \delta(k + \sum_{j=1}^5 (-1)^j k_j)  \delta (\Delta (\vec k) ) \frac1{k-k_1 + k_2 - k_3 } \prod_{j=1}^5|a(k_j)|^2 \hat f(k) \overline{\hat g (k)} dk \prod_{j=1}^5 dk_j
\end{multline*}
or at least that the sequence admits a finite number of adherence values of this form.

The behavior of $2^L t - \lfloor 2^L t\rfloor$ when $t$ is irrational is not so obvious. We recall that the closure of
\[
\{2^L t - \lfloor 2^L t\rfloor \; |\; L\in \N\}
\]
is either the torus $\R/\Z$ or a subset of null Haar measure, but that it is almost surely the torus. When $(2^L t - \lfloor 2^L t\rfloor)_L$ is dense in the torus, we believe that the sequence 
\[
(\partial_t \E(\an{P_N U_L,f} \an{g,P_M U_L } ) (t))_L
\] 
has at least an infinite number of adherence values.

For a complete description of the behavior of the sequence $2^L t  - \lfloor 2^L t \rfloor$, we refer to \cite{Douady}.
\end{remark}

\begin{remark} Given a fixed $t$, a similar argument to ours will yield that there exists a regime $\varepsilon(t,L)$ such that
\begin{multline*}
\lim_{L \rightarrow \infty} \partial_t \E(\an{P_N U_L,f} \an{g,P_M U_L } ) (t)= \\
\frac3{4\pi^4}\int_{\R^6} \delta(k + \sum_{j=1}^5 (-1)^j k_j)  \delta (\Delta (\vec k) ) \frac1{k-k_1 + k_2 - k_3 } \prod_{j=1}^5|a(k_j)|^2 \hat f(k) \overline{\hat g (k)} dk \prod_{j=1}^5 dk_j
\end{multline*}
or written differently that
\begin{multline*}
\liminf_{\varepsilon\rightarrow 0,L \rightarrow \infty}\Big| \partial_t \E(\an{P_N U_{\varepsilon,L},f} \an{g,P_M U_{\varepsilon,L} } ) (t)  \\
 - \frac3{4\pi^4}\int_{\R^6} \delta(k + \sum_{j=1}^5 (-1)^j k_j)  \delta (\Delta (\vec k) ) \frac1{k-k_1 + k_2 - k_3 } \prod_{j=1}^5|a(k_j)|^2 \hat f(k) \overline{\hat g (k)} dk \prod_{j=1}^5 dk_j\Big| = 0.
\end{multline*}
This should be compared with the result of \cite{LukSpo}.
\end{remark}

\begin{remark} This result of discontinuity tends to argue in favor of the fact that the results achieved in \cite{BGHS,CoG19,CoGe20,denghani19} are probably the best one can hope for. In particular, this goes well with the fact that the Dyson series does not converge in the cubic Schr\"odinger case, as proved in \cite{CoGe20}.
\end{remark}

\begin{remark} We restricted ourselves to the dimension 1 for seek of clarity but the techniques used do not depend on dimension, except when the dimension plays a role in the structure of first order resonances. The results could be carefully extended to higher dimension with some adaptation. \end{remark}

The proof relies on the following strategy. We first describe solutions $u_{\varepsilon,L}$ as a series (converging in Kondratiev's distributions) where each term corresponds to a certain degree in terms of Wiener chaos decomposition. Then, we describe each term of the series thanks to trees, or so-called Feynman diagrams. Finally, we analyze each of these trees and decide which of them are contributing to the considered limit.

The relevant trees are the ones with only one node, or of Wiener chaos of degree 1. The others are irrelevant for mainly two reasons : either they give a contribution of size $\varepsilon^n t^{2n}f(L)$ with $f(L) \rightarrow 0$, which gives, taking time $t$ of order $\varepsilon^{-2}$ something that goes to $0$ when $L$ goes to $\infty$; or they give a contribution of size $\varepsilon^{2n}t^mg(L)$ with $m<n$, in which case, taking as time scale $\varepsilon^{-2}$, we get a contribution of size $\varepsilon^{2(n-m)}g(L)$ and we use that $\varepsilon^{2(n-m)}$ can compensate the potentially disastrous behavior of $g(L)$.

The first case arises when, in the history of interactions between the different wavelengths, special resonances occurs. This translates as constraint equations on the wavelengths and is explained in Subsection \ref{subsec:constraint}. 

The second case arises in a more general context, which is explained in Subsection \ref{subsec:timeestimates}.

The reason we ask in $C_L(k)$ (and thus in the non-linearity) that $|\Delta(\vec k)|\geq \nu^{-1}(L)$ and then, in Assumptions \eqref{assumptionsRegime}, $\nu^{1+\alpha} = O(L^{1/2})$ is to deal with the trees presenting constraint estimates. It is probably not optimal, as explained in Remark \ref{rem:nuNotOptimal} but the proof suggests that the optimal assumption is $\nu^{1+\alpha} = o(L)$. The issue is that simply assuming $\Delta(\vec k) = 0$ ensures only that $|\Delta(\vec k)|\geq L^{-2}$. Hence, we need the condition on $\nu$ to be far enough from first order resonances.

We now explain the special regime for $\varepsilon$. If 
\[
\varepsilon^{-2}t \in \rho_t(L) + 2\pi L^2 \Z
\]
we need $\rho_t(L) = o(L)$ while $\rho_t(L)$ goes to $\infty$ to get the result. But if 
\[
\varepsilon^{-2}t = \rho_t(L) = o(L)
\]
then $\varepsilon(L)$ cannot be small enough to close the argument. So, we decided to use the degree of freedom in $2\pi L^2\Z$. To include the dyadic numbers, we chose 
\[
\varepsilon^{-2} = 2\pi L^2 2^L + \rho(L)
\]
and thus $\varepsilon$ was small enough to be able to be far from optimality regarding the trees contributing as $\varepsilon^{2n}t^mg(L)$.

Finally, the reason we ask $|k-k_1+k_2-k_3|\geq \mu^{-1}(L)$ is to be able to manipulate integrals when passing from sum to integral or when getting Dirac deltas.

\subsection{Organization of the paper}

In Section \ref{sec:Wick}, we review the definitions of Kondratiev's distributions and Wick's product. We prove global well-posedness of Equation \eqref{improvedCauchyprob}.

In Section \ref{sec:trees}, we define quintic trees and ordered quintic trees, that we use to describe the solution $u_{\varepsilon,L}$ as a sum indexed by these trees.

In Section \ref{sec:analysis}, we estimates the different contributions of the trees.

In Section \ref{sec:limits}, we take the final limits that yield to our result.

Finally, for the rest of the paper, we write $\Z/L = \frac1{L}\Z$ to lighten notations.

\subsection{Acknowledgments}

The author is supported by ANR grant ESSED ANR-18-CE40-0028.

\section{Wick's product, well-posedness of the equation}\label{sec:Wick}

\subsection{Wick's product, Kondratiev's distributions}\label{subsec:Kondratiev}

Let $\Omega, \mathcal A, \mathbb P$ be a probability onto which one can define $(\xi_k)_{k\in \Z}$ a sequence of independent Gaussian variables centered and normalized. For any $\alpha \in \N^{\Z}$ with finite support, we define 
$$
\xi_\alpha : = \prod_{k\in \Z} \frac{H_{\alpha_k}(\xi_k)}{\sqrt{\alpha_k !}},
$$
where $H_{\alpha_k}$ is the $\alpha_k$-th Hermite polynomial. It is a well-known fact that $\prod_k \sqrt{\alpha_k !}\xi_\alpha$ is the orthogonal (in $L^2$) projection of
$$
\prod_{k\in \Z} \xi_k^{\alpha_k}
$$
on the orthogonal of the polynomials of degree at most $|\alpha| - 1 = \sum_k \alpha_k -1$.

Let $\mathcal F$ be the sigma-algebra generated by the sequence $(\xi_k)_k$. We recall that for any $\phi \in L^2((\Omega, \mathcal F , \mathbb P), H^s(L\T))$, we have the decomposition (called Wiener chaos decomposition)
$$
\phi = \sum_{\alpha \in \N_f^\Z} \phi_\alpha \xi_\alpha
$$
where $\phi_\alpha \in H^s(L\T)$ and $\N_f^\Z$ is the set of sequences in $\N^\Z$ with finite support. What is more,
$$
\|\phi\|_{L^2(\Omega, H^s(L\T))}^2  = \sum_\alpha \|\phi_\alpha\|_{H^s(L\T)}^2.
$$

For more information on Wiener chaos, we refer to \cite{Janson}. For the rest of this section, we omit the dependence in $L\T$ of the Sobolev spaces.

Let $q = (q_k)_{k\in \Z} \in \C^{\Z}$, $\alpha,\beta \in \N_f^\Z$, we introduce the notations
$$
q^\alpha = \prod_k q_k^{\alpha_k} , \alpha ! = \prod_k \alpha_k ! , |\alpha| = \sum_k \alpha_k, C_{\alpha + \beta}^\alpha = \frac{(\alpha + \beta)!}{\alpha! \beta !}.
$$
Moreover, we write $(0)$ the sequence in $\N^\Z$ identically equal to $0$.

We define the Wick's product of $\xi_\alpha$ and $\xi_\beta$ as the orthogonal projection of 
$$
\xi_\alpha \xi_\beta 
$$
on the orthogonal of the polynomials of degree at most $|\alpha| + |\beta|-1$, that is 
$$
\xi_\alpha \diamond \xi_\beta =  \sqrt{C_{\alpha+\beta}^\alpha} \xi_{\alpha+\beta}.
$$

Let $(q_k)_k$ be a sequence of increasing, positive numbers such that for all $D>0$, there exists $N \in \N$ such that
$$
\sum_k \frac1{q_k^N}< \frac1{4D}.
$$

We recall that the space of Kondratiev distributions $S_{-1} (H^s)$ is defined as the inductive limit of spaces $(S_{-1,-l} (H^s))_{l\in \N}$ where $S_{-1,-l}(H^s)$ is the closure of $L^2(\Omega, H^s)$ with regards to the norm
$$
\|\phi\|_{-1,-l}^2  = \sum_\alpha \frac1{\alpha !} q^{-l\alpha} \|\phi_\alpha\|_{H^s}^2.
$$

The space of Kondratiev's distributions is the dual of the space 
\[
S_1  (H^{-s}) = \bigcap_{l\in \N} S_{1,l} (H^{-s})
\]
where $S_{1,l}(H^{-s})$ is induced by the norm
\[
\|\phi\|_{S_{1,l}}^2 = \sum_\alpha \alpha ! q^{l\alpha} \|\phi_\alpha \|_{H^{-s}}^2.
\]

For any $\phi \in S_{-1}(H^s)$, we have the chaos expansion 
$$
\phi  = \sum_\alpha \phi_\alpha \xi_\alpha
$$
therefore, we can define the Wick's product of two elements of $S_{-1}(H^s)$, $\phi$ and $\psi$ by 
$$
(\phi \diamond \psi)_\alpha = \sum_{\alpha_1 + \alpha_2  = \alpha } \sqrt{ C^{\alpha_1}_{\alpha}} \phi_{\alpha_1} \psi_{\alpha_2}.
$$
Note that for any $s>\frac12$, if $\phi \in S_{-1,-l_1}(H^s) $ and $\psi \in S_{-1,-l_2}(H^s)$, taking $l = \max(l_1 + A,l_2)$, where 
$$
\sum_k \frac1{q_k^A} < 1,
$$
we have since $H^s$ is an algebra,
\begin{multline*}
\|\phi\diamond \psi\|_{-1,-l}^2 = \sum_{\alpha} \frac1{\alpha !} q^{-l\alpha} \|(\phi \diamond \psi)_\alpha\|_{H^s}^2 =  \\
\sum_\alpha \sum_{\alpha_1,\alpha_2} \frac{q^{-l(\alpha-\alpha_1)/2}}{\sqrt{(\alpha-\alpha_1)!}}  \|\psi_{\alpha - \alpha_1}\|_{H^s} \frac{q^{-l(\alpha-\alpha_2)/2}}{\sqrt{(\alpha-\alpha_2)!}}  \|\psi_{\alpha - \alpha_2}\|_{H^s} \frac{q^{-l\alpha_1/2}}{\sqrt{\alpha_1!}}  \|\phi_{\alpha_1}\|_{H^s}  \frac{q^{-l\alpha_2/2}}{\sqrt{\alpha_2!}}  \|\phi_{\alpha_1}\|_{H^s}.
\end{multline*}
By Cauchy-Schwarz, we have 
$$
\|\phi\diamond \psi\|_{-1,-l}^2 \leq \Big( \sum_\alpha \frac{q^{-l\alpha/2}}{\sqrt{\alpha !}} \|\phi_\alpha\|_{H^s}\Big)^2 \|\psi\|_{-1,-l}^2
$$
where indeed, $\|\psi\|_{-1,-l} <\infty$. Again by Cauchy-Schwarz, we have 
$$
\|\phi\diamond \psi\|_{-1,-l}^2 \leq  \sum_\alpha q^{-(l-l_1)\alpha} \|\phi\|_{-1,-l_1}^2 \|\psi\|_{-1,-l}^2.
$$
We recognize
$$
\sum_\alpha q^{-(l-l_1)\alpha}  = \sum_{N \in \N} \Big( \sum_k \frac1{q_k^{l-l_1}} \Big)^N 
$$
which is finite since 
$$
\sum_k \frac1{q_k^{l-l_1}}
$$
is strictly less than 1.

Therefore the Wick product of 2 elements of $S_{-1}(H^s)$ is well-defined.

The space $S_{-1}(H^s)$ is called the space of Kondratiev distributions of $H^s$.

\subsection{Embeddings}

\begin{definition} Let $C_\alpha$ be defined as 
$$
C_0 = 0, C_{\alpha} = 1, \textrm{ for all }|\alpha| = 1
$$
and 
$$
C_\alpha = \sum_{\alpha_1 + \alpha_2  = \alpha }C_{\alpha_1} C_{\alpha_2}.
$$
\end{definition}

\begin{lemma} We have for all $\alpha \in \N_f^\Z$,
$$
C_\alpha \leq 4^{|\alpha|}.
$$
\end{lemma}

\begin{proof} Let $(z_k)_k \in \R^\Z$ be sequence with finite support such that
$$
\left|\sum z_k \right| < \frac14.
$$
Set $M \in \N$ and 
$$
F_M(z) = \sum_{|\alpha|\leq M}  z^\alpha C_\alpha .
$$
This is a finite sum, therefore, it converges.

We have 
$$
F_M(z) = \sum_{|\alpha| = 1}  z^\alpha + \sum_{1<|\alpha |\leq M} \sum_{\alpha_1+\alpha_2 = \alpha } z^\alpha C_{\alpha_1}C_{\alpha_2}.
$$
therefore
$$
F_M(z) = \sum_{k} z_k + F_M(z)^2 - \sum_{\mathcal A_M} C_{\alpha_1}C_{\alpha_2} z^{\alpha}
$$
where $\mathcal A_M = \{(\alpha_1,\alpha_2) | |\alpha_1|,|\alpha_2|\leq M, |\alpha_1+ \alpha_2|> M\}$. We work on the union of nonempty submanifolds of $\R^d$, $d\in \N^*$,
$$
\{z\in \R_f^{\Z} \; |\; |\sum z_k | <\frac14\textrm{ and } \left|\sum_{k} z_k - \sum_{\mathcal A_M} C_{\alpha_1}C_{\alpha_2} z^{\alpha} \right|< \frac14\}.
$$
Writing 
$$
G_M (z)=  \sum_{\mathcal C_M} C_{\alpha_1}C_{\alpha_2} z^{\alpha}
$$
we have 
$$
F_M (z) = \frac12 \pm \frac12 \sqrt{1 - 4 \sum_k z_k + 4G_M}.
$$
Because $F_M(0) = 0$, we have
$$
F_M(z) = \frac12 - \frac12 \sqrt{1 - 4 \sum_k z_k + 4G_M}.
$$
Expanding the square root, we get
$$
F_M(z) = \sum_k z_k - G_M + \sum_{n>1} \frac{(2n-3)!}{n! (n-1)!} \Big( \sum_k z_k -G_M\Big)^N.
$$
Keeping in mind that $F_M$ is a polynomial of degree at most $M$ in $z$, and that $G_M$ is a polynomial of degree at least $M+1$, we have 
$$
F_M(z) = \sum_k z_k + \sum_{1<n\leq M } \frac{(2n-3)!}{n! (n-1)!} \Big( \sum_k z_k \Big)^N.
$$
Therefore
$$
F_M(z) = \sum_k z_k + \sum_{1<n\leq M } \sum_{|\alpha|=n}  \frac{(2n-3)!}{n! (n-1)!} z^\alpha.
$$
Thus, for all $1<|\alpha| \leq M$, we have 
$$
C_\alpha = \frac{(2|\alpha|-3)!}{|\alpha|! (|\alpha|-1)!} \leq 4^{|\alpha|}.
$$
Besides, for all $|\alpha|=1$, $C_\alpha = 1 \leq 4$, which concludes the proof.
\end{proof}

\begin{definition} Let $D>0$, and $X(D)$ be the space of Kondratiev's distributions $\phi$ such that $\phi_0 = 0$ and 
$$
\sup_\alpha \frac{\|\phi_\alpha\|_{H^s}}{\sqrt{\alpha !} C_\alpha D^{|\alpha|}} < \infty
$$
endowed with the norm
$$
\|\phi\|_{X(D)} = 
\sup_\alpha \frac{\|\phi_\alpha\|_{H^s}}{\sqrt{\alpha !} C_\alpha D^{|\alpha|}}.
$$
\end{definition}

\begin{proposition} The space $X(D)$ is a Banach algebra (for the Wick product).\end{proposition}

\begin{proof}[Partial proof] We prove that $X(D)$ is complete. Let $\phi^n$ be a Cauchy sequence in $X(D)$. We have that by definition of $X(D)$, at $\alpha$ fixed, the sequence $\phi_\alpha^n$ is Cauchy in $H^s$ and thus converges towards some $\phi_\alpha$ in $H^s$. By the usual arguments, we have 
$$
\sup_\alpha \frac{\|\phi_\alpha\|_{H^s}}{\sqrt{\alpha !} C_\alpha D^{|\alpha|}}< \infty
$$
and 
$$
\sup_\alpha \frac{\|\phi_\alpha - \phi_\alpha^n\|_{H^s}}{\sqrt{\alpha !} C_\alpha D^{|\alpha|}}\rightarrow 0
$$
when $n$ goes to $\infty$. The issue at stake is to prove that $(\phi_\alpha)_\alpha$ is indeed a Kondratiev distribution. 

We have for $l$ big enough
$$
\sum_\alpha \|\phi_\alpha\|_{H^s}^2 \frac{q^{-l\alpha}}{\alpha !} \leq \sum_{\alpha} C_\alpha^2 D^{2|\alpha|} q^{-l\alpha} \|\phi\|_{X(D)}^2.
$$
Because of the bound on $C_\alpha$, we have 
$$
 \sum_\alpha \|\phi_\alpha\|_{H^s}^2 \frac{q^{-l\alpha}}{\alpha !} \leq \sum_{\alpha}  (4D)^{2|\alpha|} q^{-l\alpha} \|\phi\|_{X(D)}^2.
$$
We recognize
$$
\sum_{\alpha}  (4D)^{2|\alpha|} q^{-l\alpha} = \sum_n (4D)^{2n}\Big( \sum_k \frac1{q_k^l} \Big).
$$
For $l$ big enough
$$
\sum_k \frac1{q_k^l} < \frac1{(4D)^2}
$$
hence $(\phi_\alpha)_\alpha$ defines an element of $S_{-1,-l}(H^s)$ and thus a Kondratiev distribution.

We omit the proof that $X(D)$ is an algebra as the fact is irrelevant for the sequel and the proof is fairly straightforward.

%
\end{proof}

\subsection{Global well-posedness of the equation}

Let $X_t(D)$ be the space of  $\phi \in \mathcal C(\R, \mathcal S_{-1}(H^s))$ such that for all $t\in \R$, $\phi(t)_{(0)}= 0$ and induced by the norm :
$$
\|\phi \|_{X_t(D)} = \sup_{t\in [-T,T]}  \an{t}^{1/4}\|\phi\|_{X(D\an{t}^{1/4})}.
$$

\begin{proposition}\label{prop:wellPosedness} Let $D_0 > 0$. There exists $D > D_0$ such that for all $\phi_0 $ in the unit ball of $X(D_0 )$ with $(\phi_0)_{(0)} = 0$, the Cauchy problem
\begin{equation}\label{Cauchyprob}\left \lbrace{\begin{array}{c}
i\partial_t \phi = -\lap \phi +\varepsilon F_L(\phi) \\
  \phi(t=0) = \phi_0
  \end{array}}\right.
\end{equation}
admits a unique global solution in $X_t (D)$ and the flow thus defined is continuous in the initial datum. 
\end{proposition}

\begin{proof} We solve the fix point problem
$$
\phi (t) = A(\phi) := S(t) \phi_0 - i\int_{0}^T S(t-\tau) F_L(\phi(\tau))d\tau
$$
in the ball of $X_t(D)$ of radius $\eta $ for $D$ big enough. Since the linear flow preserves the $H^s$ norm, and since for all $t\in \R$ and $\phi$ such that $\phi_{(0)} = 0$,
\[
\an{t}^{1/4} \|\phi\|_{X(D\an{t}^{1/4}}  = \sup_{\alpha \neq (0)} \frac{\|\phi_\alpha\|_{H^s}}{\sqrt{\alpha!} C_\alpha D^{|\alpha|}\an{t}^{|\alpha|/4}} \leq \|phi\|_{X(D)},
\]
we have
$$
\an{t}^{1/4}\|A(\phi)(t)\|_{X(D\an{t}^{1/4})} \leq \|\phi_0\|_{X(D)} + |t|\an{t}^{1/4}\varepsilon \sup_{\tau \in [0,t]}\|F_L(\phi)(\tau)\|_{X(D \an{t}^{1/4})}.
$$
We have by definition
$$
\hat F_L(\phi)(k) = \frac1{(2\pi L)^2} \sum_{C_L(k)} \hat \phi(k_1) \diamond \overline{\hat \phi(k_2)} \diamond \hat \phi(k_3) \diamond \overline{\hat \phi(k_4)} \diamond \hat \phi(k_5).
$$
where
\[
C_L(k) = \{(k_1,\hdots, k_5) \in \Z/L \;| \;| k_1 - k_2 + k_3 - k_4 + k_5 = k ,\quad |\Delta (\vec k) |\geq \nu^{-1} ,\quad |k_1-k_2+k_3| \geq  \mu^{-1}\}.
\]

Therefore, for $\alpha  \in \N_f ^\Z$ such that $|\alpha|\geq 5$ (for $|\alpha|<5$, $\hat F_L(\phi)_\alpha = 0$),
\begin{multline*}
\hat F_L(\phi)_\alpha (k) = \frac1{(2\pi L)^2} \sum_{\alpha_1+ \alpha_2+ \alpha_3+ \alpha_4+ \alpha_5= \alpha } \frac{\sqrt{\alpha !}}{\sqrt{\alpha_1 ! \alpha_2 ! \alpha_3 ! \alpha_4! \alpha_5 !}}\\ \sum_{C_L(k)} 
\hat \phi_{\alpha_1}(k_1)  \overline{\hat \phi_{\alpha_2}(k_2)}  \hat \phi_{\alpha_3}(k_3)  \overline{\hat \phi_{\alpha_4}(k_4)}  \hat \phi_{\alpha_5}(k_5).
\end{multline*}
We get
\begin{multline*}
|\hat F_L(\phi)_\alpha (k)|\leq 
\frac1{(2\pi L)^2} \sum_{\alpha_1+ \alpha_2+ \alpha_3+ \alpha_4+ \alpha_5= \alpha } \frac{\sqrt{\alpha !}}{\sqrt{\alpha_1 ! \alpha_2 ! \alpha_3 ! \alpha_4! \alpha_5 !}}\\ \sum_{k_1-k_2+k_3-k_4+k_5 = k} 
\Big|\hat \phi_{\alpha_1}(k_1)  \overline{\hat \phi_{\alpha_2}(k_2)}  \hat \phi_{\alpha_3}(k_3)  \overline{\hat \phi_{\alpha_4}(k_4)}  \hat \phi_{\alpha_5}(k_5)\Big|.
\end{multline*}
By convexity of $\an{x}^{2s}$ and Cauchy-Schwarz (and using that $s>1/2$), we get
$$
\| F_L(\phi)_\alpha \|_{H^s}\lesssim   \sum_{\alpha_1+ \alpha_2+ \alpha_3+ \alpha_4+ \alpha_5= \alpha } \frac{\sqrt{\alpha !}}{\sqrt{\alpha_1 ! \alpha_2 ! \alpha_3 ! \alpha_4! \alpha_5 !}} \|\phi_{\alpha_1}\|_{H^s}\|\phi_{\alpha_2}\|_{H^s}\|\phi_{\alpha_3}\|_{H^s}\|\phi_{\alpha_4}\|_{H^s}\|\phi_{\alpha_5}\|_{H^s}.
$$
We have using the norm of $\phi$, for some constant $C_s$ depending on the Sobolev regularity,
$$
\|F_L(\phi)_\alpha(\tau) \|_{H^s}\leq C_s \sqrt{\alpha !}(\an{\tau}^{1/4}D)^{|\alpha|} C_\alpha \|\phi\|_{X(\an{\tau}^{1/4} D)}^5,
$$ 
that is
$$
\|F_L(\phi)_\alpha(\tau) \|_{H^s}\leq C_s \sqrt{\alpha !} (\an{\tau}^{1/4}D)^{|\alpha|} C_\alpha \eta^5 \an{\tau}^{-5/4}.
$$
Note that this $C_s$ does not depend on (or is uniformly bounded in) $L\in \N^*$ as its dependence in $L$ is characterized by the value of
\[
\frac1{L}\sum_{k\in \Z/L} \an{k}^{-2s} \leq \frac1{L} + 2 \int dx \an{x}^{-2s}.
\]

Because $|\alpha| \geq 5$, and $\an \tau \leq \an t$, we have 
$$
\|F_L(\phi)_\alpha(\tau) \|_{H^s}\leq C_s (\an{t}^{1/4}D)^{|\alpha|} C_\alpha \eta^5 \an{t}^{-5/4}.
$$
We deduce 
$$
\|F_L(\phi)(\tau)\|_{X(D\an{t})} \leq C_s \an{t}^{-5/4} \eta^5 .
$$
Taking $\eta$ such that $C_s\eta^4 \varepsilon \leq \frac12$, we get
$$
\|A(\phi)\|_{X_t(D)} \leq \frac{D_0}{D} \|\phi_0\|_{X(D_0)} + \frac12  \eta.
$$
Taking $D$ such that 
$$
\frac{D_0}{D} \leq \frac12 \eta 
$$
we get that the ball of $X_t(D) $ of radius $\eta $ is stable under $A$. A similar argument yields that $A$ is contracting for $\eta$ such that
\[
\tilde C_s \varepsilon \eta^4 \leq \frac12
\]
where $\tilde C_s$ is a constant independent from $L$ but depending on the Sobolev regularity $s$.  
\end{proof}

\section{Trees, description of the solution}\label{sec:trees}

\subsection{Picard expansion}

Let $\varphi$ be a bijection from $\{0,1\}\times \Z$ to $\Z$. We call, by abuse of notation,
$$
\xi_{\iota, k} = \xi_{\varphi(k)}.
$$
We set
$$
g_k = \xi_{0,k} + i\xi_{1,k}.
$$

Let, for all $x\in (L\T)^2$,
$$
a_{L} (x) = \sum_{k\in \Z} g_k \frac{e^{ikx/L}}{\sqrt{2\pi}} a(\frac{k}{L})
$$
where $a$ is a smooth function with compact support. We call
$$
u_{\varepsilon,L} 
$$
the solution to 
$$
i\partial_t u_{\varepsilon,L}  = -\lap u_{\varepsilon,L}   +\varepsilon F_L(u_{\varepsilon,L} )
$$
with initial datum $a_L$ in $X_t(D)$ for some $D$ big enough. 

First, note that 
$$
(a_{L})_\alpha = \left \lbrace{\begin{array}{cc}
\frac{e^{ikx/L}}{\sqrt L} a(k/L) & \textrm{if } |\alpha| = 1 \textrm{ and Supp}(\alpha) = \varphi(0,k)\\
i\frac{e^{ikx/L}}{\sqrt L} a(k/L) & \textrm{if } |\alpha| = 1 \textrm{ and Supp}(\alpha) = \varphi(1,k)\\
0 & \textrm{otherwise.}\end{array}} \right. .
$$
Hence
$$
\|(a_L)_\alpha\|_{H^s} \leq \sup_{x\in \R} \an{x}^s |a(x)|
$$
and therefore $a_{L}$ is in the unit ball of $X(D_0)$ for any $D_0 \geq \|\an{x}^s a\|_{L^\infty}$. Hence $u_{\varepsilon,L}$ is well-defined.

\begin{remark}\label{rem:boundSol} We have that $a_L$ belongs to some $X(D_0)$ with $D_0$ uniformly bounded in $L$. According to the proof of Proposition \ref{prop:wellPosedness}, we have that the solution 
\[
u_{\varepsilon,L}
\]
belongs to $X_t(D)$ with
\[
D = c_s D_0 \varepsilon^{1/4}
\]
where $c_s$ is a constant depending only on the Sobolev regularity $s$ and thus $D$ is uniformly bounded in $L$.

Finally, we have 
\[
\|(u_{\varepsilon,L}(t))_\alpha\|_{H^s(L\T)} \leq \sqrt{\alpha !} C_\alpha D^{|\alpha|}\an{t}^{\frac14(|\alpha| - 1)}
\]
and thus for all $\alpha \in \N_f^{\Z}$ and all $t\in \R$, we have that
\[
\|(u_{\varepsilon,L}(t))_\alpha\|_{H^s(L\T)}
\]
is uniformly bounded in $L$.
\end{remark}
\begin{proposition} The solution $u$ hence defined is, in terms of Kondratiev distributions 
$$
u_{\varepsilon,L} = \sum_{n=0}^\infty \varepsilon^n u_{n,L}
$$
where $u_{0,L}$ is $S(t)a_{L} := e^{it\Delta }a_L$ and the sequence is defined by induction by
$$
i \partial_t u_{n+1,L} = -\lap u_{n+1,L} + \sum_{n_1+n_2+n_3+n_4+n_5 = n}F_L(u_{n_1,L},u_{n_2,L},u_{n_3,L},u_{n_4,L},u_{n_5,L})
$$
with initial datum $0$.

What is more, setting that $P_N$ is the projection over chaos of degree at most $N$, we have 
$$
P_N u_{\varepsilon,L} = \sum_{n=0}^M \varepsilon^n u_{n,L}
$$
with $M = \lfloor \frac{N-1}{4} \rfloor$.
\end{proposition}

\begin{proof} We prove that the series 
$$
\sum_{n=0}^\infty \varepsilon^n u_{n,L}
$$
converges in $X_t(D)$ for $D$ big enough.

We prove that $u_{n,L}$ is either $0$ or a Wiener chaos of exact degree $4n+1$ (by which we mean a sum of monomials of exact degree $4n+1$) and that there exists $D$ big enough  and $\eta $ small enough such that $\|u_{n,L}\|_{X_t(D)} \leq \eta \leq 1$. This is true for $u_{0,L}$, we have for any $\eta > 0$,
$$
\|u_{0,L}\|_{X_t(\frac{\|\an{x}^sa\|_{L^\infty} }{\eta})} \leq \eta.
$$ 
We prove that this is true for any $n$ by induction.

We have 
$$
u_{n+1,L} (t)= -i \int_{0}^t S(t-\tau) \sum_{n_1+n_2+n_3+n_4+n_5 = n}F_L(u_{n_1}, u_{n_2}, u_{n_3}, u_{n_4}, u_{n_5})
$$
where we recall that $S(t)$ is the linear flow $e^{it\Delta}$. Since $u_{n_j}$ is of exact degree $4n_j + 1$, we get that
$$
F_L(u_{n_1}, u_{n_2}, u_{n_3}, u_{n_4}, u_{n_5})
$$
is of degree $4(n_1+n_2+n_3+n_4+n_5) + 5 = 4n + 5 = 4(n+1)+1$ or null. Besides, for any $\alpha$, we have 
$$
\|u_{n+1,L}(t)_\alpha\|_{H^s} \leq  C(\alpha !)^{1/2} |t| \sum_{n_1+n_2+n_3+n_4+n_5 = n} \sum_{\alpha_1+ \alpha_2+\alpha_3 + \alpha_4 + \alpha_5 = \alpha}\sup_{\tau \leq t}\prod_j\|u_{n_j}(\tau)_{\alpha_j}\|_{H^s} (\alpha_j !)^{-1/2}.
$$
Using that
$$
\|u_{n_j}(\tau)_{\alpha_j}\|_{H^s} (\alpha_j !)^{-1/2}
$$
is $0$ if $|\alpha_j| \neq 4 n_j + 1$ and is less that
$$
D^{4n_j+1} \an{\tau}^{n_j} C_{\alpha_j} \sqrt{\alpha_j !} \nu
$$
we get, if $|\alpha| = 4n+5$, 
$$
\|u_{n+1,L}(t)_\alpha\|_{H^s} \leq  C\nu^5(\alpha !)^{1/2} |t|  \sum_{\alpha_1+ \alpha_2+\alpha_3 + \alpha_4 + \alpha_5 = \alpha}\prod_j C_{\alpha_j} D^{4(n+1) + 1} \an{t}^{n} = C\nu^5 C_\alpha D^{4(n+1) + 1} \an{t}^{n+1}
$$
and is $0$ if $|\alpha| \neq 4n+5$.

Hence
$$
\|u_{n+1,L}\|_{X_t(D)} \leq C \nu^5.
$$
For $\nu$ small enough such that $C\nu^4 \leq 1$, we have indeed $\|u_{n+1,L}\|_{X_t(D)} \leq \nu$. Therefore, since $\nu \leq 1$ and $\varepsilon< 1$, the series 
$$
\sum_n \varepsilon u_{n,L}
$$
converges in $X_t(D)$ for $D = \frac{\|\an{x}^s a\|_{L^\infty}}{\nu}$.

Since this series satisfies equation
$$
i\partial_t u = -\lap u + \varepsilon F_L(u)
$$
with initial datum $u_{0,L}$, we have 
$$
u_{\varepsilon,L} = \sum_n \varepsilon^n u_{n,L}.
$$
What is more,
$$
P_N u_{\varepsilon,L} = \sum_n \varepsilon^n P_N u_{n,L}
$$
and because $u_{n,L}$ is of degree $4n+1$, we get $P_N u_{n,L} = u_{n,L}$ if $4n+1 \leq N$ and is $0$ otherwise.
\end{proof}

We set 
$$
\hat u_{n,L} (k,t) = \frac1{2\pi L}\int_{\T L} u_{n,L}(x,t) e^{-ikx}dx.
$$
We get that $\hat u_{n,L}(k)$ satisfies for $n\geq 1$,
\begin{multline*}
i\partial_t \hat u_{n,L} = 
k^2 \hat u_{n,L}(k) + 
\frac1{(2\pi L)^2} \sum_{n_1+n_2+n_3+n_4+n_5=n-1} \\
\sum_{C_L(k)} \hat u_{n_1,L}(k_1) \diamond \overline{\hat u_{n_2,L}(k_2)}\diamond \hat u_{n_3,L}(k_3) \diamond \overline{\hat u_{n_4,L}(k_4)}\diamond\hat u_{n_5,L}(k_5).
\end{multline*}
We set $v_{n,L}(k,t) = e^{ik^2t} \hat u_{n,L}(k,t)$ and get the equation
$$
v_{0,L}(k,t) = a(k)
$$
and 
\begin{multline}\label{eqsurvnL}
i\partial_t v_{n,L}(k) =  \frac1{(2\pi L)^2} \sum_{n_1+n_2+n_3+n_4+n_5=n-1} \\
\sum_{C_L(k)}e^{i\Delta(\vec k) t} v_{n_1,L}(k_1) \diamond \overline{v_{n_2,L}(k_2)}\diamond v_{n_3,L}(k_3) \diamond \overline{v_{n_4,L}(k_4)}\diamond v_{n_5,L}(k_5).
\end{multline}

\subsection{Quintic trees}

All the definitions in this subsection and the next one have examples in Appendix \ref{app:gloss}.

\begin{definition}\label{def:labelledTrees} Let $k\in \Z/L$. We define by induction the set $\mathcal T_n [k]$ of quintic trees with $n$ nodes by 
$$
\mathcal T_0 [k] = \{ (k)\}
$$
and 
\begin{multline*}
\mathcal T_{n+1}[k] = \{(T_1,T_2,T_3,T_4,T_5,k) \; | \;  \forall j=1, \hdots, 5, T_j \in \mathcal T_{n_j}[k_j] \textrm{ with }\\
   \sum n_j = n , (k_1,\hdots, k_5) \in C_L(k) \}.
\end{multline*}
We set
$$
l\mathcal T_n = \bigcup_{k\in \Z/L} \mathcal T_n[k].
$$
\end{definition}

\begin{definition}\label{def:functionsTrees} For any tree we define the following functions or random variables. First, we set
$$
F_{(k)}(t) = 1, g_{(k)} = g_{Lk}, A_{(k)} = a(k),  
$$
and if $T = (T_1,T_2,T_3,T_4,T_5,k)$ with $T_j \in \mathcal T_{n_j}[k_j]$, we set
$$
F_T (t) = -i\int_{0}^t e^{i\Delta \tau} F_{T_1}(\tau) \overline{F_{T_2}(\tau)} F_{T_3}(\tau) \overline{F_{T_4}(\tau)}F_{T_5}(\tau) d\tau
$$
and 
$$
g_T = g_{T_1}\diamond \overline{g_{T_2}} \diamond g_{T_3} \diamond \overline{g_{T_4 }} \diamond g_{T_5},
A_T = A_{T_1} \overline{A_{T_2}} A_{T_3} \overline{A_{T_4}}A_{T_5} .
$$
\end{definition}

\begin{definition}\label{def:labelsTrees} Finally, we define the labels $\vec T$ of the leaves of a tree $T$ in the following way
$$
\vec{(k)} = k \in \Z/L
$$
and for $T = (T_1,T_2,T_3,T_4,T_5,k)$ writing $(k_{j,1},\hdots,k_{j,4n_j+1}) = \vec T_j$ with $T_j \in \mathcal T_{n_j}[l_j]$, we set
$$
\vec T = (k_{1,1}, \hdots, k_{1,4n_1+1} ,  k_{2,1}, \hdots, k_{2,4n_2+1} ,  k_{3,1}, \hdots, k_{3,4n_3+1} ,  k_{4,1}, \hdots, k_{4,4n_4+1} ,  k_{5,1}, \hdots, k_{5,4n_5+1} ).
$$
\end{definition}

\begin{remark} The definition is consistent as $\vec T \in \R^{4n+1}$ for any $T \in l\mathcal T_n$. \end{remark}

\begin{proposition}\label{prop:descripAg} Let $T \in \mathcal T_n[k]$ and $\vec T = (k_1,\hdots, k_{4n+1})$, we have
$$
k = \sum_{j=1}^{4n+1} (-1)^{j+1} k_j, A_T = \prod_{j=1}^{4n+1} a(k_j,(-1)^{j+1}), g_T = \underset{j=1}{\overset{4n+1}{\Diamond}} g_{k_j,(-1)^{j+1}} .
$$
We used the notation
\begin{eqnarray*}
a(k_j,1) = a(k_j) & \textrm{ and } & a(k_j,-1) = \overline{a(k_j)}\\
g_{k_j,1} = g_{k_jL} & \textrm{ and } & g_{k_j,-1} = \overline{g_{k_jL}}
\end{eqnarray*}
\end{proposition}

\begin{proof} By induction starting from Equation \eqref{eqsurvnL}. \end{proof}

\begin{proposition}\label{prop:descripvn1} We have for all $(n,k,t) \in \N\times \Z/L \times \R$,
$$
v_{n,L} (k,t) = \Big[ \frac1{(2\pi L)^2} \Big]^n \sum_{T \in \mathcal T_n[k]} F_T(t) A_T g_T.
$$
\end{proposition}

\begin{proof} By induction. \end{proof}

We introduce unlabelled quintic trees.

\begin{definition}\label{def:unlabelledtree} Let $\mathcal T_n$ defined by induction by
$$
\mathcal T_0 = \{\bot \}
$$
and 
$$
\mathcal T_{n+1} = \{(T_1,T_2,T_3,T_4,T_5) | \forall j=1,\hdots , 5, T_j \in \mathcal T_{n_j} \sum_{j=1}^5 n_j = n\}.
$$
Given $\vec k = (k_1,\hdots, k_{4n+1}) \in (\Z/L)^{4n+1}$, we write 
$$
\bot (\vec k) = (k_1)
$$
and 
$$
(T_1,T_2,T_3,T_4,T_5) (\vec k) = (T_1(\vec k_1),T_2(\vec k_2),T_3(\vec k_3), T_4(\vec k_4)  , T_5(\vec k_5),k)
$$
with $T_j \in \mathcal T_{n_j}$ and $\vec k_j = (k_{\tilde n_j + 1}, \hdots, k_{\tilde n_j + 4n_j + 1})$ with 
$$
\tilde n_j = \sum_{l=1}^{j-1} (4n_l +1) \quad \textrm{and}\quad k = -\sum_{j=1}^{4n+1} (-1)^j k_n.
$$
\end{definition}

We now give a definition of labels for the nodes and the leaves that helps seeing the "history" of the tree, as the paternity of nodes. 

\begin{definition}\label{def:SetOrder} For any tree $T\in \bigcup_n  \mathcal T_n$ we define by induction
$$
N(\bot) = \emptyset, \tilde N(\bot) = \{0\}
$$
and if $T = (T_1,T_2,T_3,T_4,T_5)$, we write
$$
N(T)= \{0\}\sqcup ( \{1\}\times N(T_1) )\sqcup (\{2\}\times N(T_2) )\sqcup (\{3\}\times N(T_3)) \sqcup (\{4\}\times N(T_4)) \sqcup (\{5\}\times N(T_5))
$$
and
$$
 \tilde N(T)= \{0\}\sqcup (\{1\}\times \tilde N(T_1)) \sqcup (\{2\}\times \tilde N(T_2) )\sqcup (\{3\}\times \tilde N(T_3)) \sqcup (\{4\}\times \tilde N(T_4)) \sqcup (\{5\}  \times \tilde N(T_5)).
$$
\end{definition}

\begin{definition}\label{def:functionLabels} Let $\vec k = (k_1,\hdots,k_{4n+1}) \in (\Z/L)^{4n+1}$. We define 
$$
k_{T,\vec k} : \tilde N(T) \rightarrow \Z/L
$$
such that $k_{\bot,\vec k} (0) = k_1$ and if $T = (T_1,T_2,T_3,T_4,T_5)$ with $T_j \in \mathcal T_{n_j}$ and $\vec k_j = (k_{\tilde n_j + 1}, \hdots, k_{\tilde n_j + 4n_j + 1})$ with 
$$
\tilde n_j = \sum_{l=1}^{j-1} (4n_l +1)
$$
we define 
$$
k_{T,\vec k} (0)= \sum_{j=1}^{4n+1} (-1)^{j+1} k_j \quad \textrm{and} \quad k_{T,\vec k}((j,m)) = k_{T_j,\vec k_j} (m).
$$  

We define also
$$
\Omega_{T,\vec k} : N(T) \rightarrow \R
$$
such that 
$$
\Omega_{T,\vec k}(0) = k^2_{T,\vec k}(0) - k^2_{T,\vec k}((1,0)) + k^2_{T,\vec k}((2,0)) - k^2_{T,\vec k}((3,0)) + k^2_{T,\vec k}((4,0)) - k^2_{T,\vec k}((5,0))
$$ 
and 
$$
\Omega_T ((j,l)) = (-1)^{j+1} \Omega_{T_j} (l).
$$
\end{definition}

We use the notation $\Omega$ instead of $\Delta$ because of the discussion of the sign.

Finally, we define a condition on $\vec k$ such that $T(\vec k) \in \mathcal T_n[k].$

\begin{definition} We define by induction $C_T(k)$ by
$$
C_{\bot} (k) = \{k\}
$$
and $\vec k \in C_T(k)$ iff
$$
k_{T,\vec k}(0) = k, \quad |k_{T,\vec k}(1,0) - k_{T,\vec k}(2,0) + k_{T,\vec k}(3,0) - k_{T,\vec k}(0)|\geq \mu^{-1} , 
$$
and
$$
|\Omega_{T,\vec k}(0)| \geq \nu^{-1},\quad \forall j=1,\hdots, 5, \vec k_j  \in C_{T_j}(k_{T,\vec k}(j,0))
$$
with $T = (T_1,\hdots,T_5)$, $T_j \in \mathcal T_{n_j}$ and 
$$
\vec k_j = (k_{\tilde n_j + 1,\hdots, k_{\tilde n_j+ 4n_j+1}})
$$
with $\tilde n_j = \sum_{l<j} (4n_l+1)$.  
\end{definition}

\begin{proposition}\label{prop:descripTnk} We have 
$$
\mathcal T_n[k] = \{T(\vec k) | T\in \mathcal T_n \textrm{ and } \vec k \in C_T(k)\}.
$$
\end{proposition}

\begin{proof} By induction using Definition \ref{def:unlabelledtree}.\end{proof}

\begin{proposition}\label{prop:descripvn2} We have for all $(n,k,t) \in \N\times \Z/L \times \R$,
$$
v_{n,L} (k,t) = \Big[ \frac1{(2\pi L)^2} \Big]^n \sum_{T \in \mathcal T_n} \sum_{\vec k \in C_T(k)} F_{T(\vec k)}(t) A_{\vec k} g_{\vec k}
$$
where 
$$
A_{\vec k} = \prod_{j=1}^{4n+1} a(k_j,(-1)^{j+1}) \quad \textrm{and} \quad g_{\vec k} = \underset{j=1}{\overset{4n+1}{\Diamond}} g_{k_j,(-1)^{j+1}}.
$$
\end{proposition}

\begin{proof} Direct consequence of Propositions \ref{prop:descripvn1}, \ref{prop:descripTnk}, and \ref{prop:descripAg}.
\end{proof}

\subsection{Ordered trees}

\begin{definition}\label{def:orderonnodes}
We define on $N(T)$ the partial order relation $R_T$ such that for all $l\in N(T)$,
$$
j\;R_T\; 0
$$
for all $l_1 = (j_1,m_1) \in \{j_1\}\times N(T_{j_1}) $ and $l_2 = (j_2,m_2) \in \{j_2\}\times N(T_{j_2})$, we have 
$$
l_1 R_T l_2 \Leftrightarrow j_1 = j_2 \textrm{ and } m_1 R_{T_{j_1}} m_2.
$$
\end{definition}

\begin{remark} The partial order $R_T$ represents parenthood in the tree. \end{remark}

\begin{proposition}\label{prop:ordreetFT} The cardinal of $N(T)$ is $n$. We have 
$$
F_{T(\vec k)}(t) = (-i)^{N_T}\int_{I_T(t)} \prod_{l\in N(T)} e^{it_l \Omega_{T,\vec k}(l)}dt_l
$$
where $I_T(t) = \{(t_l)_{l\in N(T)} \in [0,t]^n \; |\; l_1 Rl_2 \Rightarrow t_{l_1} \leq t_{l_2}\}$ and where $N_T$ is defined by induction on the number of nodes of $T$ by $N_\bot = 0$ and for $T = (T_1,T_2,T_3,T_4,T_5)$,
\[
N_T = 1 - \sum_{j=1}^5 (-1)^j N_{T_j}.
\]
\end{proposition}

\begin{lemma}\label{lem:ordreetIT} We have 
$$
I_T(t) = \{(t_l)_{l\in N(T)}\in [0,t]^n \;|\;  \forall j=1\hdots 5 , \quad (t_{(j,l_j)})_{l_j\in N(T_j)} \in I_{T_j}(t_0) \}.
$$
\end{lemma}

\begin{proof} We recall that $0$ is bigger than any other element of $N(T)$ and that if $j_1$ is different from $j_2$, $(j_1,m_1)$ and $(j_2,m_2)$ are not comparable while $(j,m_1)R (j,m_2)$ iff $m_1 R_{T_j} m_2$.
\end{proof}

\begin{proof}[Proof of Proposition \ref{prop:ordreetFT}] From Lemma \ref{lem:ordreetIT}, the induction follows almost directly, it remains to see is that
$$
\overline{e^{i\Delta t}} = e^{-i\Delta t}
$$
which justifies in the definition of the function $\Omega_T(j,m)$ the sign $-1$ when $j$ is even, as in
$$
\Omega_T(j,m) = -\Omega_{T_j}(m).
$$
\end{proof}

\begin{definition} Let $T \in \mathcal T_n$, we define $\mathfrak S_T$ be the set of bijection from $[1,n]\cap \N$ to $N(T)$such that $\varphi \in \mathfrak S_T$ if and only if
\[
\forall l_1,l_2 \in N(T) ,\quad l_1 R_T l_2 \Rightarrow \varphi^{-1}(l_1)\leq \varphi^{-1}(l_2).
\]
We set
\[
F^\varphi_{T,\vec k}(t) = \int_{0\leq t_1\leq \hdots \leq t_n\leq t} \prod_{j=1}^n e^{i\Omega_{T,\vec k}(j) t_j}dt_j.
\]
\end{definition}

\begin{proposition}We have 
\[
F_{T(\vec k)}(t) = (-i)^{N_T}\sum_{\varphi \in \mathfrak S_T} F_{T,\vec k}^\varphi (t).
\]
\end{proposition}

\begin{proof} Description of $I_T(t)$.\end{proof}

%
%
%
%
%
%

\subsection{Description of the solution}

\begin{lemma} Let $\vec k_1 \in (\Z/L)^{4n_1+1} $ and $\vec k_2 \in (\Z/L)^{4n_2+1}$. We have 
$$
\E( \overline{g_{\vec k_1}}g_{\vec k_2}) = 0
$$
unless $n_1 = n_2 = :n$ and there exists a bijection $\sigma \in \mathfrak S_{4n+1}$ such that for all $j=1,\hdots, 4n+1$, 
$$
k^1_{\sigma (j)} = k^2_j.
$$
Besides, if for all $j_1\neq j_2$, $k_{j_1}^1 \neq k_{j_2}^2 $ then $\sigma$ is uniquely defined and conserves the parity.
\end{lemma}

\begin{proof}We have 
$$
g_{\vec k_l} = \underset{j=1}{\overset{4n_l+1}{\Diamond}} g_{k_j^l,(-1)^{j+1}}
$$
that is by definition of  $g_{k_j^l,(-1)^{j+1}}$,
$$
g_{\vec k_l} = \underset{j=1}{\overset{4n_l+1}{\Diamond}} \Big( \xi_{0,Lk^l_j} + i (-1)^{j+1} \xi_{1,Lk^l_j}\Big).
$$
Expanding the product, we get
$$
g_{\vec k_l} = \sum_{\iota : [1,4n_l +1]\cap \N \rightarrow \{0,1\} } \underset{j=1}{\overset{4n_l+1}{\Diamond}} \xi_{\iota(j),Lk_j^l} i^{\iota (j) } (-1)^{(j+1)\iota (j)}.
$$
Therefore,
$$
\E (\overline{g_{\vec k_1}}g_{\vec k_2} ) = \sum_{\iota_1,\iota_2} \E \Big[  \underset{j_1=1}{\overset{4n_1+1}{\Diamond}} \xi_{\iota_1(j_1),Lk_{j_1}^1} i^{-\iota_1(j_1)}(-1)^{(j_1+1)\iota_1(j_1)} \underset{j_2=1}{\overset{4n_2+1}{\Diamond}} \xi_{\iota_2(j_2),Lk_{j_2}^2} i^{-\iota_2(j_2)}(-1)^{(j_2+1)\iota_2(j_2)}\Big].
$$
First, for 
$$
E \Big[  \underset{j_1=1}{\overset{4n_1+1}{\Diamond}} \xi_{\iota_1(j_1),Lk_{j_1}^1} i^{-\iota_1(j_1)}(-1)^{(j_1+1)\iota_1(j_1)} \underset{j_2=1}{\overset{4n_2+1}{\Diamond}} \xi_{\iota_2(j_2),Lk_{j_2}^2} i^{-\iota_2(j_2)}(-1)^{(j_2+1)\iota_2(j_2)}\Big]
$$
not to be zero, one needs the existence of a bijection from $[1,4n_2+1]\cap \N$ to $[1, 4n_1+1]\cap \N$ such that
$$
k_{\sigma(j_2)}^1 = k_{j_2}^2.
$$
This implies that $n_1 = n_2$ but $\sigma$ is not necessarily uniquely defined unless the $k_j^1$ are all different. In this case, we get
$$
\E (\overline{g_{\vec k_1}}g_{\vec k_2} ) = \sum_{\iota_1,\iota_2} \E \Big[  \underset{j_1=1}{\overset{4n_1+1}{\Diamond}} \xi_{\iota_1(j_1),Lk_{j_1}^1} i^{-\iota_1(j_1)}(-1)^{(j_1+1)\iota_1(j_1)} \underset{j_2=1}{\overset{4n_2+1}{\Diamond}} \xi_{\iota_2\circ \sigma(j_2),Lk_{j_2}^1} i^{-\iota_2\circ \sigma (j_2)}(-1)^{(\sigma(j_2)+1)\iota_2\circ\sigma(j_2)}\Big].
$$
For 
$$
\E \Big[  \underset{j_1=1}{\overset{4n_1+1}{\Diamond}} \xi_{\iota_1(j_1),Lk_{j_1}^1} i^{-\iota_1(j_1)}(-1)^{(j_1+1)\iota_1(j_1)} \underset{j_2=1}{\overset{4n_2+1}{\Diamond}} \xi_{\iota_2\circ \sigma(j_2),Lk_{j_2}^1} i^{-\iota_2\circ \sigma (j_2)}(-1)^{(\sigma(j_2)+1)\iota_2\circ\sigma(j_2)}\Big]
$$
not to be $0$, we need $\iota_2 \circ \sigma = \iota_1$, therefore, setting $k_j = k_j^1$ and $n = n_1 = n_2$, we get
$$
\E (\overline{g_{\vec k_1}}g_{\vec k_2} ) = \sum_{\iota} \E \Big[ \Big| \underset{j=1}{\overset{4n+1}{\Diamond}} \xi_{\iota(j),Lk_{j}}\Big|^2 \Big]\prod(-1)^{(j+\sigma(j))\iota(j)} .
$$
If $\sigma$ conserves parity, we have 
$$
\E (\overline{g_{\vec k_1}}g_{\vec k_2} ) = \sum_{\iota} \E \Big[ \Big| \underset{j=1}{\overset{4n+1}{\Diamond}} \xi_{\iota(j),Lk_{j}}\Big|^2 \Big] > 0 .
$$
Otherwise, let $j_0$ such that $j_0 + \sigma(j_0)$ is odd. For any $\iota$, we write $\iota'$ such that $\iota'(j_0) = -\iota (j_0)$ and $\iota'(j) = \iota(j)$ for all $j\neq j_0$. Since $\iota \mapsto \iota'$ is a bijection of 
$$
[1,4n+1]\cap \N \rightarrow \{0,1\},
$$
we have
$$
\E (\overline{g_{\vec k_1}}g_{\vec k_2} ) = \sum_{\iota} \E \Big[ \Big| \underset{j=1}{\overset{4n+1}{\Diamond}} \xi_{\iota'(j),Lk_{j}}\Big|^2 \Big] \prod(-1)^{(j+\sigma(j))\iota'(j)} .
$$
Since the $k_j$ are all different replacing $\xi_{\iota(j_0),Lk_{j_0}}$ by $\xi_{\iota'(j_0),Lk_{j_0}}$ consists in exchanging the roles of $\xi_{0,Lk_{j_0}}$ and $\xi_{1,Lk_{j_0}}$. In other words, since 
$$
(\xi_{\iota(j),k_j})_j \textrm{ and } (\xi_{\iota'(j),k_j})_j
$$
have the same law
$$
\E \Big[ \Big| \underset{j=1}{\overset{4n+1}{\Diamond}} \xi_{\iota'(j),k_{j}}\Big|^2 \Big] =\E \Big[ \Big| \underset{j=1}{\overset{4n+1}{\Diamond}} \xi_{\iota(j),Lk_{j}}\Big|^2 \Big] .
$$
Since 
$$
 \prod_j (-1)^{(j+\sigma(j))\iota'(j)} =  -\prod_j(-1)^{(j+\sigma(j))\iota'(j)}
 $$
 we get
$$
\E (\overline{g_{T_1}}g_{T_2} ) = -\E (\overline{g_{T_1}}g_{T_2} ) =0.
$$
\end{proof}

\begin{proposition} We have 
$$
\E ( \overline{v_{n_1,L}(k,t)} v_{n_2,L}(k,t)) = 0
$$
unless $n_1 = n_2=n$ and in this case
\begin{multline*}
\E ( \overline{v_{n,L}(k,t)} v_{n,L}(k,t)) =\\
 \frac1{(2\pi L)^{4n}} \sum_{(T_1,T_2) \in \mathcal T_n^2}(-i)^{N_{T_2}-N_{T_1}} \sum_{\sigma \in \mathfrak S_{4n+1}}\sum_{\varphi_j \in \mathfrak{S}_{T_j}} \sum_{\vec k \in C(T_1,T_2,\sigma,k)}  G_{T_1,T_2,\vec k, \vec k_\sigma}^{\varphi_1,\varphi_2} (t) \overline{A_{\vec k}} A_{\vec k_\sigma} \E(\overline{g_{\vec k}} g_{\vec k_\sigma})
\end{multline*}
where setting $\vec k = (k_1, \hdots, k_{4n+1})$, we used the notations
$$
k_\sigma  = (k_{\sigma(1)},\hdots, k_{\sigma(4n+1)}),
$$
$$
C(T_1,T_2,\sigma, k) = \{\vec k \in C_{T_1}(k)\; |\; \vec k_{\sigma} \in C_{T_2}(k) \}
$$
and 
\[
 G_{T_1,T_2,\vec k, \vec k_\sigma}^{\varphi_1,\varphi_2} = \overline{F_{T_1,\vec k}^{\varphi_1}} F_{T_2,\vec k_\sigma}^{\varphi_2}.
\]
\end{proposition}

\begin{corollary}\label{cor:descripdtE} We have 
\begin{multline*}
\partial_t \E ( | v_{n,L}(k,t)|^2) = \\
\frac1{(2\pi L)^{4n}} \sum_{(T_1,T_2) (-i)^{N_{T_2}-N_{T_1}}\in \mathcal T_n^2} \sum_{\sigma \in \mathfrak S_{4n+1}}\sum_{\varphi_j \in \mathfrak{S}_{T_j}} \sum_{\vec k \in C(T_1,T_2,\sigma,k)}  \partial_t G_{T_1,T_2,\vec k, \vec k_\sigma}^{\varphi_1,\varphi_2} (t) \overline{A_{\vec k}} A_{\vec k_\sigma} \E(\overline{g_{\vec k}} g_{\vec k_\sigma})
\end{multline*}
\end{corollary}

\begin{remark} Note that since $a$ is of compact support, all the sums are finite and thus converge.\end{remark}

\section{Analysis}\label{sec:analysis}

\subsection{Time estimates}\label{subsec:timeestimates}

\begin{proposition}\label{prop:GMpair} Let $\Omega_1, \hdots , \Omega_M \in \Z^*/L^2$, we define 
$$
G_M (t) = \int_{0\leq t_1\leq \hdots \leq t_M \leq t } \prod_{j=1}^M e^{i\Omega_j t_j} dt_j.
$$
We have 
$$
G_M(t) = O_M (L^{2M} |t|^{\lfloor \frac{M-1}{2}\rfloor})
$$
unless $M$ is even and for all $j=1$ to $M/2$, we have $\Omega_{2j-1} + \Omega_{2j} = 0$, in which case 
$$
G_M(t) = \prod_{j=1}^{M/2} \frac{1}{\Omega_{2j}} \frac{t^{M/2}}{(M/2)!} + O (L^{2M} |t|^{\lfloor \frac{M-1}{2}\rfloor}).
$$
\end{proposition}

\begin{proof} We prove it by induction. If $M=0$, we have 
$$
G_M(t) = 1 = \prod_{j=1}^{0} \frac{1}{\Omega_{2j}} \frac{t^0}{0!} + O (L^{0} |t|^{\lfloor \frac{M-1}{2}\rfloor}).
$$

We assume that the proposition is true for all $m\leq M$, we prove it for $M+1$.

We have if $M=0$,
$$
G_{M+1} (t) = \frac{e^{i\Omega_1 t}-1}{i\Omega_1} = O(L^2).
$$

Otherwise, since
$$
G_{M+1} (t) = \int_{0}^t e^{i\Omega_{M+1} t_{M+1} } G_M(t_{M+1}) dt_{M+1},
$$
by integration by parts, we get
$$
G_{M+1}(t) = \frac{e^{i\Omega_{M+1} t }}{i\Omega_{M+1}} G_M(t) - \int_{0}^t \frac{e^{i(\Omega_{M+1} + \Omega_M)\tau }}{i\Omega_{M+1}} G_{M-1}(\tau) d\tau.
$$

First case, $M+1$ is odd. We have since $M$ is even
$$
|G_M(t)| = O_M(t^{M/2} L^M ) + O_M(t^{(M-2)/2 L^{2M}} = O_M(L^{2M} t^{M/2}),
$$
therefore
$$
\frac{e^{i\Omega_{M+1} t }}{i\Omega_{M+1}} G_M(t)  = O_M(t^{M/2} L^{2(M+1)}).
$$
Since $M-1$ is even, we have
$$
\frac{e^{i(\Omega_{M+1} + \Omega_M)\tau }}{i\Omega_{M+1}} G_{M-1}(\tau) = O_M (\tau^{(M-2)/2}L^{2(M-1)})
$$
thus
$$
\int_{0}^t d\tau \frac{e^{i(\Omega_{M+1} + \Omega_M)\tau }}{i\Omega_{M+1}} G_{M-1}(\tau) = O_M(t^{M/2} L^{2M})
$$
which concludes the induction when $M+1$ is odd.

Second case, $M+1$ is even. We have 
$$
G_M(t) = O_M(L^{2M}t^{(M-1)/2})
$$
hence 
$$
\frac{e^{i\Omega_{M+1} t }}{i\Omega_{M+1}} G_M(t)  = O_M(L^{2(M+1)} t^{(M-1)/2})
$$
Then, if $\Omega_{M+1} + \Omega_M \neq 0$, we have either $M=1$ and then
$$
\int_{0}^t d\tau \frac{e^{i(\Omega_{M+1} + \Omega_M)\tau }}{i\Omega_{M+1}} G_{M-1}(\tau) = \frac{e^{i(\Omega_{M+1} + \Omega_M)t} -1}{-\Omega_{M+1}(\Omega_{M+1} + \Omega_M)} = O(L^4)
$$
or $M> 1$ and by integration by parts
\begin{multline*}
\int_{0}^t \frac{e^{i(\Omega_{M+1} + \Omega_M)\tau }}{i\Omega_{M+1}} G_{M-1}(\tau) d\tau = \\
\frac{e^{i(\Omega_{M+1} + \Omega_M)t }} {-(\Omega_{M+1} + \Omega_M)\Omega_{M+1}} G_{M-1}(t)+ \int_{0}^t d\tau \frac{e^{i(\Omega_{M+1} + \Omega_M)t }} {-(\Omega_{M+1} + \Omega_M)\Omega_{M+1}} e^{i\Omega_{M-1}\tau} G_{M-2}(\tau).
\end{multline*}
We have 
$$
\frac{e^{i(\Omega_{M+1} + \Omega_M)t }} {-(\Omega_{M+1} + \Omega_M)\Omega_{M+1}} G_{M-1}(t) = O_M(L^4 L^{2(M-1)} t^{(M-1)/2} )
$$
and 
$$
 \frac{e^{i(\Omega_{M+1} + \Omega_M)t }} {-(\Omega_{M+1} + \Omega_M)\Omega_{M+1}} e^{i\Omega_{M-1}\tau} G_{M-2}(\tau) = O_M(L^4 L^{2(M-2)} t^{(M-3)/2})
 $$
 hence
 $$
 G_{M+1}(t) = O_M(L^{2(M+1)}t^{(M-1)/2}).
 $$
 
If $\Omega_{M+1} + \Omega_M = 0$, either at least $\Omega_{2j-1} + \Omega_{2j} \neq 0$ for one $j=1$ to $(M-1)/2$, in which case
$$
\int_{0}^t d\tau \frac{e^{i(\Omega_{M+1} + \Omega_M)\tau }}{i\Omega_{M+1}} G_{M-1}(\tau) = \int_{0}^t O_M (L^2 L^{2(M-1)} \tau^{(M-3)/2}) = O_M(L^{2(M+1)} t^{(M-1)/2})
$$
or $ \Omega_{2j-1} + \Omega_{2j} = 0$ for all $j=1$ to $(M-1)/2$ (which includes $M=1$), in which case
\begin{multline*}
\int_{0}^t d\tau \frac{e^{i(\Omega_{M+1} + \Omega_M)\tau }}{i\Omega_{M+1}} G_{M-1}(\tau) = \int_{0}^t \frac1{i\Omega_{M+1}} \Big( \prod_{j=1}^{(M-1)/2} \frac{1}{i\Omega_{2j}} \frac{\tau^{(M-1)/2}}{(M-1)/2!} + O_M (L^{2(M-1)} \tau^{(M-3)/2})\Big) = \\
\prod_{j=1}^{(M+1)/2} \frac{1}{i\Omega_{2j}} \frac{t^{(M+1)/2}}{(M+1)/2!} + O_M(L^{2(M+1)} t^{(M-1)/2}).
\end{multline*}
\end{proof}

\begin{proposition}\label{prop:GMimpair} Let $M$ be odd. We are in one of the following cases : \begin{enumerate}
\item $M=1$ then
$$
G_M(t) = \frac{e^{i\Omega_M t}-1}{i\Omega_M};
$$
\item we are not in case 1, $\forall  j= 1, \hdots , \frac{M-1}{2}$, $\Omega_{2j} = -\Omega_{2j-1} = -\Omega_M$ then
$$
G_M(t) = (i\Omega_M)^{-(M+1)/2}  (-1)^{(M-1)/2}\frac{t^{(M-1)/2}}{(M-1)/2 ! }(e^{i\Omega_M t}-1) + O(L^{2M} t^{(M-3)/2});
$$
\item we are not in cases 1 or 2, $\forall j = 1,\hdots, \frac{M-1}{2}$, $\Omega_{2j} + \Omega_{2j-1} = 0$ then
$$
G_M(t) =(-1)^{(M-1)/2} \prod_{j=0}^{(M-1)/2} \frac1{i\Omega_{2j+1}} e^{i\Omega_M t}  \frac{t^{(M-1)/2}}{(M-1)/2 ! } + O(L^{2M} t^{(M-3)/2});
$$
\item we are not in cases 1, 2, or 3, $\forall j = 1,\hdots, \frac{M-1}{2}$, $\Omega_{2j} + \Omega_{2j+1} = 0$ then 
$$
G_M(t) = (-1)^{(M+1)/2}\prod_{j=0}^{(M-1)/2} \frac1{i\Omega_{2j+1}}   \frac{t^{(M-1)/2}}{(M-1)/2 ! } + O(L^{2M} t^{(M-3)/2});
$$
\item we are not in cases 1, 2, 3 or 4, there exists $j_0 \in [1,\frac{M-1}{2}] \cap \N$ such that $\Omega_{2j_0+1} + \Omega_{2j_0} + \Omega_{2j_0-1} = 0 $ and $\forall j>j_0$, $\Omega_{2j+1} + \Omega_{2j} = 0$ and $\forall j<j_0$, $\Omega_{2j} + \Omega_{2j-1} = 0$ then
$$
G_M(t) = (-1)^{(M-1)/2}\prod_{j=0}^{(M-1)/2} \frac1{i\Omega_{2j+1}}   \frac{t^{(M-1)/2}}{(M-1)/2 ! } + O(L^{2M} t^{(M-3)/2});
$$
\item we are not in cases 1, 2, 3, 4, or 5, then
$$
G_M(t) = O(L^{2M} t^{(M-3)/2}).
$$
\end{enumerate}
\end{proposition}

Before proving the proposition, we prove the following lemma.

\begin{lemma} Let $n \in \N$, and $\alpha \neq  0$, we have 
\begin{equation}\label{Exactintegral}
\int_{0}^t e^{i\alpha \tau} \tau^n d\tau = \frac{e^{i\alpha t}}{i\alpha} \sum_{k=0}^n \Big( -\frac{1}{i\alpha}\Big)^k \frac{n!}{(n-k)!} t^{n-k} + \Big( -\frac1{i\alpha}\Big)^{n+1} n! .
\end{equation}
\end{lemma}

\begin{proof} Writing
$$
I_n = \int_{0}^t e^{i\alpha \tau} \tau^n d\tau
$$
the proof follows from
$$
I_0 = \frac{e^{i\alpha t -1}}{i\alpha}
$$
and the induction relation
$$
I_{n+1} = \frac{e^{i\alpha t}}{i\alpha} t^n  - \frac{n}{i\alpha} I_{n-1}.
$$
\end{proof}

\begin{proof}[Proof of Proposition \ref{prop:GMimpair}]We omit in the proof the dependance in $M$ of the constant. We proceed by induction over $M$. 

If $M=1$, we are in case 1. 

We assume that the proposition is true up to $M-2$, and we prove it for $M$. 

\textbf{Case 2} We have by definition
$$
G_M(t) = \int_{0}^t dt_M e^{i\Omega_M t_M}\int_{0\leq t_1\leq \hdots \leq t_{M-1}\leq t_M} \prod_{n=1}^{M-1}e^{i\Omega_nt_n} dt_n.
$$
Because of the hypothesis on the $\Omega_n$, we have 
$$
G_M(t) = \int_{0}^t dt_M e^{i\Omega_M t_M}\int_{0\leq t_1\leq \hdots \leq t_{M-1}\leq t_M} \prod_{n=1}^{(M-1)/2}e^{i\Omega_M(t_{2n-1} -t_{2n})} dt_{2n-1}dt_{2n}.
$$
By integration by parts, we get
$$
G_M(t) = \frac{e^{i\Omega_M t}}{i\Omega_M} G_{M-1}(t) - \int_{0}^t \frac{1}{i\Omega_M} G_{M-2(\tau)} d\tau.
$$
By Proposition \ref{prop:GMpair}, we have 
$$
\frac{e^{i\Omega_M t}}{i\Omega_M} G_{M-1}(t) = (-1)^{(M-1)/2} \frac1{(i\Omega_M)^{(M+1)/2}} \frac{t^{(M-1)/2}}{(M-1)/2!} +O(t^{(M-3)/2} L^{2M}).
$$
We have that $G_{M-2}$ is either in case 1 or 2, hence
$$
G_{M-2} (\tau) = (-1)^{(M-3)/2}(e^{i\Omega_M \tau} - 1)  \frac1{(i\Omega_M)^{(M-1)/2}} \frac{\tau^{(M-3)/2}}{(M-3)/2!} + O(\tau^{(M-5)/2} L^{2(M-2)}).
$$
We use \eqref{Exactintegral} to get
$$
G_M(t) = \frac{t^{(M-1)/2}}{(M-1)/2!} (-1)^{(M-1)/2}( e^{i\Omega_M t} -1) (i\Omega_M)^{-(M+1)/2}+ O(t^{(M-3)/2} L^{2M}),
$$
which is the desired result.

\textbf{Case 3} By integration by parts, we have 
$$
G_M(t) = \frac{e^{i\Omega_M t}}{i\Omega_M} G_{M-1} (t) - \int_{0}^t \frac{e^{i(\Omega_M + \Omega_{M-1}) \tau}}{i\Omega_M} G_{M-2}(\tau)  d\tau.
$$
By Proposition \ref{prop:GMpair}, we have 
$$
G_{M-1}(t) = (-1)^{(M-1)/2}\prod_{j=0}^{(M-3)/2} \frac1{i\Omega_{2j+1}} \frac{t^{(M-1)/2}}{(M-1)/2 ! } + O(t^{(M-3)/2} L^{2(M-1)})
$$
therefore
\begin{multline*}
G_M(t) = \\
(-1)^{(M-1)/2}e^{i\Omega_M t}\prod_{j=0}^{(M-1)/2} \frac1{i\Omega_{2j+1}} \frac{t^{(M-1)/2}}{(M-1)/2 ! }   - \int_{0}^t \frac{e^{i(\Omega_M + \Omega_{M-1}) \tau}}{i\Omega_M} G_{M-2}(\tau)  d\tau + O(t^{(M-3)/2}L^{2M}).
\end{multline*}
We have that $G_{M-2}$ is either in case 1, 2 or 3. If it is in case 1 or 2, then $ \Omega_M  +\Omega_{M-1} \neq 0 $ otherwise $G_M$ is in case 2 and $\Omega_M + \Omega_{M-1} + \Omega_{M-2} = \Omega_M \neq 0$, therefore, 
$$
G_M(t) = e^{i\Omega_M t}\prod_{j=0}^{(M-1)/2} \frac1{i\Omega_{2j+1}} \frac{t^{(M-1)/2}}{(M-1)/2 ! }  - \int_{0}^t \frac{e^{i(\Omega_M + \Omega_{M-1}) \tau}}{i\Omega_M} \frac{e^{i\Omega_{M-2}\tau} - 1}{i\Omega_{M-2}}  d\tau + O(t^{(M-3)/2}L^{2M})
$$
if $G_{M-2}$ is in case 1 and 
\begin{multline*}
G_M(t) = e^{i\Omega_M t}\prod_{j=0}^{(M-1)/2} \frac1{i\Omega_{2j+1}} \frac{t^{(M-1)/2}}{(M-1)/2 ! } \\
 - \int_{0}^t \frac{e^{i(\Omega_M + \Omega_{M-1}) \tau}}{i\Omega_M}(-1)^{(M-3)/2} \frac{e^{i\Omega_{M-2}\tau} - 1}{(i\Omega_{M-2})^{(M-1)/2}}\frac{\tau^{(M-3)/2}}{(M-3)/2!}  d\tau +  O(t^{(M-3)/2}L^{2M})
\end{multline*}
if $G_{M-2}$ is in case 2. In both cases, by \eqref{Exactintegral}, we get
$$
G_M(t) = e^{i\Omega_M t}\prod_{j=0}^{(M-1)/2} \frac1{i\Omega_{2j+1}} \frac{t^{(M-1)/2}}{(M-1)/2 ! } + O(t^{(M-3)/2}L^{2M}) 
$$
which is the desired result.

If $G_{M-2}$ is in case 3, then
$$
G_{M-2}(t) = (-1)^{(M-3)/2}e^{i\Omega_{M-2}t} \prod_{j=0}^{(M-1)/2} \frac1{i\Omega_{2j+1}}  \frac{t^{(M-3)/2}}{(M-3)/2 ! } + O(t^{(M-5)/2}L^{2M}).
$$
Since $\Omega_M + \Omega_{M-1} + \Omega_{M-2} = \Omega_M \neq 0$, we have by \eqref{Exactintegral}
$$
G_M(t) = e^{i\Omega_M t}\prod_{j=0}^{(M-1)/2} \frac1{i\Omega_{2j+1}} \frac{t^{(M-1)/2}}{(M-1)/2 ! } + O(t^{(M-3)/2}L^{2M}) 
$$

\textbf{Case 4} The integration by parts yields
$$
G_M(t) = \frac{e^{i\Omega_M t}}{i\Omega_M} G_{M-1} (t) - \int_{0}^t \frac{1}{i\Omega_M} G_{M-2}(\tau)  d\tau.
$$
Because we are not in case 2 or 3, by Proposition \ref{prop:GMpair}, we have
$$
G_{M-1} (t) = O(t^{(M-3)/2 }L^{2(M-1)})
$$
thus
$$
G_M(t) =- \int_{0}^t \frac{1}{i\Omega_M} G_{M-2}(\tau)  d\tau + O(t^{(M-3)/2}L^{2M}) .
$$
We have that $G_{M-2}$ is either in case 1, 2, or 4. If $G_{M-2} $ is in case 1 then
$$
 - \int_{0}^t \frac{1}{i\Omega_M} G_{M-2}(\tau)  d\tau = - \int_{0}^t \frac{e^{i\Omega_{M-2} \tau} -1}{i\Omega_M i \Omega_{M-2}}  = \frac{t}{i\Omega_M i \Omega_{M-2}} + O(L^{6})
 $$
which yields the desired result.

If $G_{M-2}$ is in case 2, then 
$$
G_M(t) =  (-1)^{(M-1)/2} \int_{0}^t \frac{1}{i\Omega_M} \frac{e^{i\Omega_{M-2}\tau} - 1}{(i\Omega_{M-2})^{(M-3)/2}}  \frac{\tau^{(M-3)/2}}{(M-3)/2 ! }  d\tau +O(t^{(M-3)/2}L^{2M}) .
$$
By \eqref{Exactintegral}, we get
$$
G_M(t) = (-1)^{(M+1)/2}  \frac{1}{i\Omega_M} \frac{  1}{(i\Omega_{M-2})^{(M-3)/2}}  \frac{t^{(M-1)/2}}{(M-1)/2 ! } + O(t^{(M-3)/2}L^{2M})  .
$$

If $G_{M-2}$ is in case 4, then
$$
G_M(t) = - (-1)^{(M-1)/2} \int_{0}^t \frac{1}{i\Omega_M} \prod_{j=0}^{(M-3)/2} \frac1{i\Omega_{2j+1}} \frac{\tau^{(M-3)/2}}{(M-3)/2 ! }  d\tau +  O(t^{(M-3)/2}L^{2M})
$$
therefore
$$
G_M(t) =  (-1)^{(M+1)/2} \prod_{j=0}^{(M-1)/2} \frac1{i\Omega_{2j+1}} \frac{t^{(M-1)/2}}{(M-1)/2 ! } + O(t^{(M-3)/2}L^{2M})   .
$$

\textbf{Case 5} Let $j_0$ such that $\Omega_{2j_0 + 1} + \Omega_{2j_0}  + \Omega_{2j_0 - 1} = 0$. We cannot have $ \Omega_{2j_0}  + \Omega_{2j_0 - 1} = 0$ hence
$$
G_{M-1}(t) = O(t^{M-3)/3}L^{2(M-1)})
$$
by Proposition \ref{prop:GMpair}.

\textbf{Case 5.1 :} $j_0 = \frac{M-1}{2}$. We recall that by integration by parts
$$
G_M (t) = \frac{e^{i\Omega_M t}}{i\Omega_M} G_{M-1}(t) - \int_0^t \frac{e^{i(\Omega_M + \Omega_{M-1} )\tau}}{i\Omega_M} G_{M-2}(\tau) d\tau.
$$
We have that $G_{M-2}$ is either in case 1, 2, or 3. If $G_{M-2}$ is either in case 1 or 2, we have 
$$
G_{M-2}(\tau) =(-1)^{(M-3)/2} \frac{\tau^{(M-3)/2}}{(M-3)/2 !} \frac{e^{i\Omega_{M-2}}-1}{(i\Omega_{M-2})^{(M-1)/2}} + O( \tau^{(M-5)/2} L^{2(M-2)}).
$$
Therefore by \eqref{Exactintegral}, since $\Omega_M +\Omega_{M-1} = -\Omega_{M-2} \neq 0$ and $\Omega_M + \Omega_{M-1} + \Omega_{M-2} = 0$, we have
$$
G_M(t) = (-1)^{(M-1)/2} \frac1{i\Omega_M (i\Omega_{M-2})^{(M-1)/2}} \frac{\tau^{(M-1)/2}}{(M-1)/2 ! } + O(t^{(M-3)/2 L^{2M}}).
$$
If $G_{M-2}$ is in case 3, we have 
$$
G_M(t) = - \int_{0}^t \frac{e^{i(\Omega_M + \Omega_{M-1} + \Omega_{M-2} )\tau}}{i\Omega_M} (-1)^{(M-3)/2} \prod_{j=0}^{(M-3)/2} \frac1{i\Omega_{2j+1} } \frac{\tau^{(M-3)/2}}{(M-3)/2 !} + O(t^{(M-3)/2 L^{2M}}).
$$
And since $\Omega_M + \Omega_{M-1} + \Omega_{M-2} = 0$, we get
$$
G_M(t) =  (-1)^{(M-1)/2} \prod_{j=0}^{(M-1)/2} \frac1{i\Omega_{2j+1} } \frac{t^{(M-1)/2}}{(M-1)/2 !} + O(t^{(M-3)/2 L^{2M}}).
$$
This is the desired result in case 5.1.

\textbf{Case 5.2} $j_0 \neq \frac{M-1}{2}$. In this case, $\Omega_M + \Omega_{M-1} = 0$ and $G_{M-2}$ is in case 5. Therefore
$$
G_M(t) = -\int_{0}^t \frac{d\tau}{i\Omega_M} (-1)^{(M-3)/2} \prod_{j=0}^{(M-3)/2} \frac1{i\Omega_{2j+1}} \frac{\tau^{(M-3)/2}}{(M-3)/2!} + O(t^{(M-3)/2} L^{2M})
$$
which yields by exact computation
$$
G_M (t) = (-1)^{(M-1)/2}  \prod_{j=0}^{(M-1)/2} \frac1{i\Omega_{2j+1}} \frac{t^{(M-1)/2}}{(M-1)/2!} + O(t^{(M-3)/2} L^{2M}).
$$

\textbf{Case 6} Because we are not in cases 1, 2, or 3, we have 
$$
G_{M-1} (t) = O(t^{(M-3)/2} L^{2(M-1)}).
$$
If $G_{M-2}$ is in case 1 or 2, we have $\Omega_M + \Omega_{M-1} + \Omega_{M-2} \neq 0 $ otherwise $G_M$ is in case 5 (and even 5.1) and $\Omega_M + \Omega_{M-1} \neq 0$ otherwise $G_M$ is in case 2 or 4. Therefore by integration by parts, induction hypothesis on $G_{M-2}$ and \eqref{Exactintegral}, we get
$$
G_M(t) =  O(t^{(M-3)/2} L^{2M}).
$$

If $G_{M-2}$ is in case 3, then $\Omega_M + \Omega_{M-1} + \Omega_{M-2} \neq 0$ otherwise $G_M$ is in case 5 and we conclude by the usual strategy.

If $G_{M-2}$ is in case 4, then $\Omega_M + \Omega_{M-1} \neq 0$ otherwise $G_M$ is in case 4 and we conclude by the usual strategy.

If $G_{M-2}$ is in case 5, then $\Omega_M + \Omega_{M-1} \neq 0$ otherwise $G_M$ is in case 5 and we conclude by the usual strategy.

If $G_{M-2}$ is in case 6, then
$$
G_{M-2}(\tau) = O(\tau^{(M-5)/2} L^{2(M-2)})
$$
and we can conclude.

\end{proof}

\subsection{Constraint estimates}\label{subsec:constraint}

Given the description of 
$$
\partial_t E(|v_n|^2)
$$
in Corollary \ref{cor:descripdtE}, we separate the sum in three parts, in the case $n>1$, either \textbf{Case A} $\sigma, \varphi, T_1, T_2$ are such that for all $\vec k \in C(T_1,T_2,\sigma,k)$, we have that
$$
\partial_t G_{T_1,T_2,\vec k,\vec k_\sigma}^\varphi \neq O(t^{n-2}L^{4n})
$$
or \textbf{Case B} $\sigma, \varphi$ are not such but however
$$
\partial_t G_{T_1,T_2,\vec k,\vec k_\sigma}^\varphi \neq O(t^{n-2}L^{4n})
$$
or \textbf{Case C} 
$$
\partial_t G_{T_1,T_2,\vec k,\vec k_\sigma}^\varphi = O(t^{n-2}L^{4n}).
$$

We first explain why \textbf{Case A} never happens.

\begin{proposition}\label{prop:nocaseA} There does not exist any $(T_1,T_2,\sigma,k)$ such that for all $\vec k \in C(T_1,T_2,\sigma,k)$, we have that
$$
\partial_t G_{T_1,T_2,\vec k,\vec k_\sigma}^\varphi \neq O(t^{n-2}L^{4n})
$$
\end{proposition}

\begin{proof} Let $T_1, T_2\in \mathcal T_n $, $\sigma \in \mathfrak{S}_{4n+1} $, $\varphi_1 \in \mathfrak S_{T_1}$, $\varphi_2 \in \mathfrak{S}_{T_2}$ and $k\in \Z/L$. Assume that for all $\vec k \in C(T_1,T_2,\sigma,k)$, we have 
\[
\partial_t G_{T_1,T_2,\vec k,\vec k_\sigma}^\varphi \neq O(t^{n-2}L^{4n})
\] 
and let us prove that it yields to a contradiction.

We recall that
\[
\partial_t G_{T_1,T_2,\vec k,\vec k_\sigma}^\varphi = \overline{\partial_t F_{T_1,\vec k}^{\varphi_1}} \; F_{T_2,\vec k_\sigma}^{\varphi_2} +  \overline{ F_{T_1,\vec k}^{\varphi_1}} \; \partial_t F_{T_2,\vec k_\sigma}^{\varphi_2}.
\]

\textbf{Case 1} : the integer $n$ is even. 

The maximal order for $t$ in $F_{T_2,\vec k_\sigma}^{\varphi_2}$ is $\frac{n}2$ and for $t$ in $\overline{\partial_t F_{T_1,\vec k}^{\varphi_1}}$ is $\frac{n-2}{2}$ for a total order of
\[
n-1.
\]
The situation for the other term in $\partial_t G$ is symmetric.

Lower orders are 
\[
\mathcal O (L^{4n-2}t^{n-2}).
\]

For $F_{T_2,\vec k_\sigma}^{\varphi_2}$ to be of order $\frac{n}2$ in $t$, one needs to be in the situation that for $j=1$ to $\frac{n}{2}$,
\[
\Omega_{T_2,\vec k_\sigma} (\varphi_2(2j-1)) + \Omega_{T_2,\vec k_\sigma} (\varphi_2(2j)) = 0.
\]
But $\Omega_{T_2,\vec k_\sigma} (\varphi (2j-1))$ and $\Omega_{T_2,\vec k_\sigma} (\varphi_2(2j))$ are two different second order polynomials in $\vec k$ and therefore their sum cannot be identically $0$ for all $\vec k \in C(T_1,T_2,\sigma,k)$, which yields to a contradiction.

\textbf{Case 2} : the integer $n$ is odd. 

The maximal order for $t$ in $F_{T_2,\vec k_\sigma}^{\varphi_2}$ is $\frac{n-1}2$ and for $t$ in $\overline{\partial_t F_{T_1,\vec k}^{\varphi_1}}$ is $\frac{n-1}{2}$ for a total order of
\[
n-1.
\]
The situation for the other term in $\partial_t G$ is symmetric.

Lower orders are 
\[
\mathcal O (L^{4n-2}t^{n-2}).
\]

For $\overline{\partial_t F_{T_1,\vec k}^{\varphi_1}}$ to be of order $\frac{n-1}2$ in $t$, one needs to be in the situation that for $j=1$ to $\frac{n-1}{2}$,
\[
\Omega_{T_1,\vec k} (\varphi_1 (2j-1)) + \Omega_{T_1,\vec k}(\varphi_1(2j)) = 0.
\]
But $\Omega_{T_1,\vec k}(\varphi_1 (2j-1))$ and $\Omega_{T_1,\vec k}(\varphi_1(2j))$ are two different second order polynomials in $\vec k$ and therefore their sum cannot be identically $0$ for all $\vec k \in C(T_1,T_2,\sigma,k)$, which yields to a contradiction.
\end{proof}

\begin{proposition}\label{prop:CaseB} Let $T_1, T_2\in \mathcal T_n $, $\sigma \in \mathfrak{S}_{4n+1} $, $\varphi_1 \in \mathfrak S_{T_1}$, $\varphi_2 \in \mathfrak{S}_{T_2}$ and $k\in \Z/L$. Set
\[
V(T_1,T_2,\sigma,k,\varphi_1,\varphi_2) = \frac1{(4\pi L)^n}\sum_{\vec k \in C(T_1,T_2,\sigma,k)} \partial_t G_{T_1,T_2,\vec k, \vec k_\sigma}^{\varphi_1,\varphi_2} (t) \overline{A_{\vec k}} A_{\vec k_\sigma} \E(\overline{g_{\vec k}}g_{\vec k_\sigma}).
\]
We have for all $\alpha >0$,
\[
V(T_1,T_2,\sigma,k,\varphi_1,\varphi_2) = O_{a,T_1,T_2,\varphi_1,\varphi_2,\sigma,\alpha} \Big[ t^{n-1} \nu^{(1+\alpha)n/2} L^{-(n-1)/2}\Big] + O_{a,n} (L^{4n-2}t^{n-2}).
\]
\end{proposition}

\begin{proof} In all the proof, we omit the dependence in $T_1,T_2,\sigma,k,\varphi_1,\varphi_2$. 

We write 
\[
\underline{C} = \{\vec k \in C \; |\; \overline{\partial_t F_{T_1,\vec k}^{\varphi_1}}F_{T_2,\vec k_\sigma}^{\varphi_2} \neq O_n (L^{4n-2}t^{n-2})\}
\]
and 
\[
\tilde C = \{\vec k \in C \; |\; \overline{ F_{T_1,\vec k}^{\varphi_1}}\partial_t F_{T_2,\vec k_\sigma}^{\varphi_2} \neq O_n (L^{4n-2}t^{n-2})\}.
\]
We set
\[
W = \frac1{(4\pi L)^n}\sum_{\vec k \in \underline C}\overline{\partial_t F_{T_1,\vec k}^{\varphi_1}}F_{T_2,\vec k_\sigma}^{\varphi_2} \overline{A_{\vec k}} A_{\vec k_\sigma} \E(\overline{g_{\vec k}}g_{\vec k_\sigma})
\] 
and
\[
W' = \frac1{(4\pi L)^n}\sum_{\vec k \in \tilde C}\overline{ F_{T_1,\vec k}^{\varphi_1}}\partial_t F_{T_2,\vec k_\sigma}^{\varphi_2} \overline{A_{\vec k}} A_{\vec k_\sigma} \E(\overline{g_{\vec k}}g_{\vec k_\sigma})
\]

Since for all $\vec k \in C$, 
\[
\Big| \overline{A_{\vec k}} A_{\vec k_\sigma} \E(\overline{g_{\vec k}}g_{\vec k_\sigma})\Big| \leq (4n+1)!\prod_{j=1}^{4n+1} |a(k_j)|^2 
\]
and since $\prod_{j=1}^{4n+1} |a(k_j)|^2 $ is integrable on 
\[
\{ \vec k \; |\; \sum_{j=1}^{4n+1}k_j (-1)^{j+1} = k\}
\]
we get that
\[
V = W + W' + \mathcal O_{n,a} (L^{4n-2}t^{n-2}).
\]
The estimate on $W'$ being symmetric, we only estimate $W$.

\textbf{Case 1} : the integer $n$ is even. 

Let $\vec k\in \underline C$, we have 
\[
F_{T_2,\vec k_\sigma}^{\varphi_2} = \frac{t^{n/2}}{(n/2)!} \prod_{j=1}^{n/2} \frac1{i\Omega_{T_2,\vec k_\sigma} (\varphi_2(2j))} + O_n(L^{2n}t^{n/2-1})
\]
and 
\[
\Big| \partial_t F_{T_1,\vec k}^{\varphi_1}\Big| \leq 2 \frac{t^{(n-2)/2}}{((n-2)/2)!} \prod_{j=0}^{(n-2)/2} \frac1{|\Omega_{T_1,\vec k} (\varphi_1(2j+1))|} +  O_n(L^{2n-2}t^{n/2-2}).
\]
What is more, $\vec k$ belongs to $\underline C^2 \cap \underline C^1$ with 
\[
\underline C^2 = \{\vec k \in C \; |\; \forall j=1,\hdots, \frac{n}{2}, \quad \Omega_{T_2,\vec k_\sigma } (\varphi_2(2j-1)) + \Omega_{T_2,\vec k_\sigma}(\varphi_2(2j)) = 0\}
\]
and 
\[
\underline C^1 = \underline{C_1}^1 \cup \underline{C_2}^1 \cup \underline{C_3}^1
\]
where
\[
\underline{C_1}^1 = \{\vec k \in C \; |\; \forall j=1,\hdots, \frac{n-2}{2}, \quad \Omega_{T_1,\vec k }(\varphi_1 (2j-1)) + \Omega_{T_1,\vec k}(\varphi_1(2j)) = 0\},
\]
\[
\underline{C_2}^1 = \{\vec k \in C \; |\; \forall j=1,\hdots, \frac{n-2}{2}, \quad \Omega_{T_1,\vec k } (\varphi_1(2j+1)) + \Omega_{T_1,\vec k}(\varphi_1(2j)) = 0\},
\]
and finally
\[
\underline{C_3}^1 = \left \lbrace \vec k \in C\; |\; \exists j_0,
\begin{array}{c} \forall j=j_0+1,\hdots, \frac{n-2}{2},\quad  \Omega_{T_1,\vec k }(\varphi_1 (2j+1)) + \Omega_{T_1,\vec k}(\varphi_1(2j)) =0 \\
\Omega_{T_1,\vec k }(\varphi_1 (2j_0+1) )+ \Omega_{T_1,\vec k}(\varphi_1(2j_0))+\Omega_{T_1,\vec k}(\varphi_1(2j_0-1)) = 0\\
\forall j = 1,\hdots , j_0-1, \quad \Omega_{T_1,\vec k }(\varphi_1 (2j-1)) + \Omega_{T_1,\vec k}(\varphi_1(2j)) = 0
\end{array}  \right\rbrace .
\]
The sets $\underline{C_1}^1$, $\underline{C_2}^1$ and $\underline{C_3}^1$ corresponds to the different cases in Proposition \ref{prop:GMimpair}.

Because of the estimate on $\overline{\partial_t F_{T_1,\vec k}^{\varphi_1}}F_{T_2,\vec k_\sigma}^{\varphi_2}$, we get
\begin{multline*}
W \leq \\
2\frac{(4n+1)!}{((n-2)/2)!(n/2)!}t^{n-1} \frac1{(4\pi L)^n}\sum_{\vec k \in \underline C}  \prod_{j=0}^{(n-2)/2} \frac1{|\Omega_{T_1,\vec k}(\varphi_1 (2j+1))|}\prod_{j=1}^{n/2} \frac1{|\Omega_{T_2,\vec k_\sigma} (\varphi_2(2j))|} \prod_{j=1}^{4n+1}|a(k_j)|^2 \\
+ \mathcal O_{n,a} (L^{4n-2}t^{n-2}).
\end{multline*}

By Cauchy-Schwarz inequality, we get
\[
W \leq C_n t^{n-1}W_1 W_2 + \mathcal O_{n,a} (L^{4n-2}t^{n-2})
\]
where $C_n$ is a constant depending only on $n$,
\[
W_1^2 = \frac1{(4\pi L)^n}\sum_{\vec k \in \underline C^1}  \prod_{j=0}^{(n-2)/2} \frac1{|\Omega_{T_1,\vec k}(\varphi_1 (2j+1))|^2} \prod_{j=1}^{4n+1}|a(k_j)|^2
\]
and 
\[
W_2^2 = \frac1{(4\pi L)^n}\sum_{\vec k \in \underline C^2}  \prod_{j=1}^{n/2} \frac1{|\Omega_{T_2,\vec k_\sigma}(\varphi_2 (2j))|^2} \prod_{j=1}^{4n+1}|a(k_j)|^2.
\]

We estimate $W_2$. For each $j=1$ to $n$, and for each $m=1$ to $5$, we write
\[ 
l_m(j) = k_{T_2,\vec k_\sigma} (s_m(\varphi_2 (j))
\]
where $s_m : N(T_2) \rightarrow \tilde N(T_2)$ is defined by induction by $s_m (0) = (m,0)$ and $s_m(m_1,l) = (m_1,s_m(l))$. In other words $s_m(j)$ is the label of the $m$th subnode or leaf of the node indexed by $j$. 

We also write
\[ 
l(j) = l_1(j) - l_2(j) + l_3(j) - l_4(j) + l_5(j) \quad \textrm{and} \quad \bar l(j) = l(j) - l_1(j) + l_2(j)-l_3(j).
\]

We have that $S_{T_2,\sigma,\varphi_2} : \vec k \mapsto \vec l = ((l_1(j),l_2(j),l_3(j),l_4(j))_{1\leq j\leq n}$ is linear and injective. Indeed, if $((l_1(j),l_2(j),l_3(j),l_4(j))_{1\leq j\leq n}$ is fixed, then, since $l(n) = k_{T_2,\vec k_{\sigma}}(0) = k$, we get that $l_5(n)$ is fixed. Now, since $l(n-1)$ is one of the $l_m(n)$ for $m=1$ to $5$, we get that $l_5(n-1)$ is fixed. By going down the tree, we get to know the full $(l_m(j))_{m,j}$ for $m=1$ to $5$ and $j =1$ to $n$. In particular we know the labels of the leaves for all the nodes at the bottom of the tree. We know the labels of the leaves, in other words, $\vec k_\sigma$ and knowing $\sigma$, we know $\vec k$. 

We set $\Omega_j(\vec l) = \Omega_{T_2,\vec k_\sigma}(\varphi_2(j))$. 

The image of $\underline C^2$ by $S_{T_2,\sigma,\varphi_2}$ is included in 
\begin{multline*}
\underline S^2 = \{ \vec l \in (\Z/L)^{4n} \; |\; \forall j=1,\hdots, n ,\quad |\bar l(j)|\geq \mu^{-1}, \\
 |\Omega_j(\vec l)|\geq \nu^{-1} ,\quad \forall j=1,\hdots,\frac{n}{2}, \quad \Omega_{2j-1}(\vec l) + \Omega_{2j}(\vec l) = 0\}.
\end{multline*}

We get, with the new notations
\[
W_2^2 \leq \frac{\nu^{(1+\alpha)n/2)}}{(2\pi L)^{4n}} \sum_{\vec l \in \underline S^2} \prod_{j=1}^{n/2} \frac1{|\Omega_{2j}(\vec l)|^{1-\alpha}} \prod_{j=1}^{4n+1} |a(S_{T_2,\sigma,\varphi_2}(\vec l)_j|^2
\]
with the convention $S_{T_2,\sigma,\varphi_2}(\vec l)_j = k_j$. Because $a$ has compact support, so has $a\circ S_{T_2,\sigma}$ and therefore, there exists $K = K(a,\sigma,T_2,\varphi_2)$ such that 
\[
W_2^2 \leq \max |a|^{2(4n+1)} \frac{\nu^{(1+\alpha)n/2)}}{(2\pi L)^{4n}} \sum_{\vec l \in \underline S^2 \cap [-K,K]^{4n}} \prod_{j=1}^{n/2} \frac1{|\Omega_{2j}(\vec l)|^{1-\alpha}}.
\]

We count the degrees of freedom in $\underline S^2$. We start by fixing $l_1(n),l_2(n),l_3(n),l_4(n)$. As we remarked earlier, this fixes automatically $l_5(n)$ and therefore, $l(n-1)$, and $\Omega_n(\vec l)$. Because we are in $\underline S^2$, this fixes too $\Omega_{n-1}(\vec l)$. 

We then fix arbitrarily $l_1(n-1),l_2(n-1)$ and $l_3(n-1)$, which fixes automatically $\bar l(n-1)$. We recall that $\Omega_{n-1}(\vec l)$ is fixed and that
\[
\Omega_{n-1}(\vec l) = l^2(n-1) - l_1^2(n-1) + l_2^2(n-1) - l_3^2(n-1) - \bar l(n-1) - 2\bar l(n-1) l_4(n-1).
\]
Therefore, this fixes $l_4(n-1)$ and in turn $l_5(n-1)$. 

Going down the tree, we get that fixing $(l_1(2j),l_2(2j),l_3(2j),l_4(2j), l_1(2j-1),l_2(2j-1),l_3(2j-1))$ for $j=n/2$ to $1$ is sufficient to recover the whole $\vec l$. Therefore, we have 
\[
W_2^2\leq C_a \frac{\nu^{(1+\alpha)n/2)}}{(2\pi L)^{4n}} \sum_{\vec l \in (\Z/L \cap [-K,K])^{7n/2}} \prod_{j=1}^{n/2} \frac1{|\Omega_{2j}(\vec l)|^{1-\alpha}}.
\]
Because
\[
\Omega_{2j}(\vec l) = -2 \bar l(2j) \Big( l_4(2j) - \frac1{2\bar l(2j)}(l^2(2j) - l_1^2(2j)+ l_2^2(2j) - l_3^2(2j)-\bar l^2(2j)\Big),
\]
By integrating in the following order $l_3(1),l_2(1),l_1(1),l_4(2),l_3(2),l_2(2),l_1(2),\hdots, l_3(n-1),l_2(n-1),l_1(n-1), l_4(n),l_3(n),l_2(n),l_1(n)$, we get that
\[
 \prod_{j=1}^{n/2} \frac1{|\Omega_{2j}(\vec l)|^{1-\alpha}}
\]
is integrable on compacts and therefore
\[
W_2^2 \lesssim_{a,T_2,\sigma,\varphi_2}  \frac{\nu^{(1+\alpha)n/2)}}{L^{n/2}}.
\]

For $W_1$, there are in each subset forming $\underline C^1$, $\frac{n-2}{2}$ constraint estimates on the $\Omega$s, while we integrate $\frac{n}2$ different $\Omega^{-2}$s. Therefore, by applying the same strategy as for $W_2$, we get
\[
W_1^2\lesssim_{a,T_1,\varphi_1}  \frac{\nu^{(1+\alpha)n/2}}{L^{(n-2)/2}}.
\]

We deduce 
\[
W = \mathcal O_{a,T_1,T_2,\sigma,\varphi_1,\varphi_2,n}\Big(   \frac{\nu^{(1+\alpha)n/2)}}{L^{(n-1)/2}}t^{n-1}\Big) + \mathcal O_{n,a}(L^{4n-2}t^{n-2})
\]
which yields the result when $n$ is even.

\textbf{Case 2 :} $n$ is odd.

Let $\vec k \in \underline C$. We have
\[
|F_{T_2,\vec k_\sigma}^{\varphi_2}| \leq 2 \frac{t^{(n-1)/2}}{((n-1)/2)!} \prod_{j=0}^{(n-1)/2}\frac1{|\Omega_{T_2,\vec k_\sigma}(\varphi_2(2j+1))|} + \mathcal O_n{L^{2n}t^{(n-1)/2}-1}
\]
and 
\[
|\partial_t F_{T_1,\vec k}^{\varphi_1} | \leq  \frac{t^{(n-1)/2}}{((n-1)/2)!} \prod_{j=1}^{(n-1)/2}\frac1{|\Omega_{T_1,\vec k}(\varphi_1(2j))|} + \mathcal O_n{L^{2n}t^{(n-1)/2}-1}.
\]
What is more, $\vec k$ belongs to $\underline C^1 \cap \underline C^2$ with
\[
\underline C^1 = \{ \vec k \in C \; | \; \forall j=1,\hdots ,\frac{n-1}{2}, \quad \Omega_{T_1,\vec k}(\varphi_1(2j-1)) + \Omega_{T_1,\vec k}(\varphi_1(2j)) = 0\}
\]
and 
\[
\underline C^2 = \underline{C_1}^2 \cup \underline{C_2}^2 \cup \underline{C_3}^2
\]
where
\[
\underline{C_1}^2 = \{ \vec k \in C \; | \; \forall j=1,\hdots ,\frac{n-1}{2}, \quad \Omega_{T_2,\vec k_\sigma}(\varphi_2(2j-1)) + \Omega_{T_2,\vec k_\sigma}(\varphi_2(2j)) = 0\},
\]
\[
\underline{C_2}^2 = \{ \vec k \in C \; | \; \forall j=1,\hdots ,\frac{n-1}{2}, \quad \Omega_{T_2,\vec k_\sigma}(\varphi_2(2j+1)) + \Omega_{T_2,\vec k_\sigma}(\varphi_2(2j)) = 0\},
\]
and finally
\[
\underline{C_3}^2 = \left \lbrace \vec k \in C \; | \; \exists j_0, \quad \begin{array}{c}
\forall j=j_0+ 1,\hdots ,\frac{n-1}{2}, \quad \Omega_{T_2,\vec k_\sigma}(\varphi_2(2j+1)) + \Omega_{T_2,\vec k_\sigma}(\varphi_2(2j)) = 0,\\
\Omega_{T_2,\vec k_\sigma}(\varphi_2(2j_0-1)) + \Omega_{T_2,\vec k_\sigma}(\varphi_2(2j_0))+\Omega_{T_2,\vec k_\sigma}(\varphi_2(2j_0+1))=0,\\
\forall j=1,\hdots ,j_0-1, \quad \Omega_{T_2,\vec k_\sigma}(\varphi_2(2j-1)) + \Omega_{T_2,\vec k_\sigma}(\varphi_2(2j)) = 0
\end{array}\right \rbrace .
\]

We get as previously by Cauchy-Schwarz inequality,
\[
W\leq C_n W_1 W_2 t^{n-1} + \mathcal O_{n,a}(L^{4n-2} t^{n-2})
\]
with
\[
W_1^2 = \frac1{(2\pi L)^{4n}} \sum_{\vec k \in \underline C^1}\prod_{j=1}^{(n-1)/2}\frac1{|\Omega_{T_1,\vec k}(\varphi_1(2j))|} \prod_{j=1}^{4n+1}|a(k_j)|^2
\]
and
\[
W_2^2 = \frac1{(2\pi L)^{4n}} \sum_{\vec k \in \underline C^1}\prod_{j=0}^{(n-1)/2}\frac1{|\Omega_{T_2,\vec k_\sigma}(\varphi_2(2j+1))|} \prod_{j=1}^{4n+1}|a(k_j)|^2.
\]

We repeat the same strategy as in the case $n$ even.

For $W_1$, there are $\frac{n-1}2$ $|\Omega|^{-2}$ to integrate and $\frac{n-1}{2}$ constraints estimates. Hence,
\[
W_1^2 \lesssim_{a,T_1,\varphi_1,\alpha} \frac{\nu^{(1+\alpha)(n-1)/2}}{L^{(n-1)/2}}
\]

For $W_2$, there are $\frac{n+1}2$ $|\Omega|^{-2}$ to integrate and $\frac{n-1}{2}$ constraint equations, hence
\[
W_1^2 \lesssim_{a,T_2,\varphi_2,\alpha,\sigma} \frac{\nu^{(1+\alpha)(n+1)/2}}{L^{(n-1)/2}}.
\]

Therefore, we have
\[
W = \mathcal O_{a,T_1,T_2,\varphi_1,\varphi_2,\sigma,\alpha}\Big( \frac{\nu^{(1+\alpha)n/2}}{L^{(n-1)/2}}\Big) + \mathcal O_{a,n} (L^{4n-2}t^{n-2})
\]
which concludes the proof when $n$ is odd.

\end{proof}

\begin{remark}\label{rem:nuNotOptimal} The application of the Cauchy-Schwarz inequality prevents us from being optimal. The worst case scenario we can think of ($n$ odd, $T_1=T_2$, $\varphi_1 = \varphi_2$, $\sigma=Id$) yielding a bound of the form
\[
\frac{\nu^{(1+\alpha)(n-1)/2}}{L^{(n-1)/2}}\nu^\alpha.
\]
\end{remark}

\subsection{Case \texorpdfstring{$n=1$}{n=1}}

We now deal with the case $n=1$.

The set $\mathcal T_1$ is reduced to $1$ element $T = (\bot,\bot,\bot,\bot,\bot)$. We write $C(k) = C_T(k)$ and for all $\sigma \in \mathfrak S_5$, we write $C_\sigma (k) = C(T,T,\sigma)$. We also write $\mathfrak A $ the set of $\sigma \in \mathfrak S_5$ that conserves parity.

\begin{proposition}\label{prop:n=1} We have 
$$
\partial_t \E(|v_1(t,k)|^2) = \sum_{\sigma \in \mathfrak A} \sum_{\vec k \in C_\sigma(k)} \frac{2}{(2\pi L)^4} \frac{\sin(\Delta(\vec k) t )}{\Delta (\vec k)} \prod_{j=1}^5|a(k_j)|^2 + O_a(L^{-1}\nu).
$$
\end{proposition}

\begin{proof} By definition, we have 
$$
\partial_t \E(|v_1(t,k)|^2) = 2\re \Big[ \sum_{\sigma \in \mathfrak S_5} \sum_{\vec k \in C_\sigma(k)} \frac1{(2\pi L)^4} \frac{e^{i\Delta(\vec k) t} -1 }{i\Delta (\vec k)}e^{-i\Delta(\vec k_\sigma) t} A_{\vec k} \overline{A_{\vec k_\sigma}} \E(g_{\vec k} \overline{g_{\vec k_\sigma}})\Big].
$$

We write
$$
\partial_t \E(|v_1(t,k)|^2) = A + B
$$
with
$$
A = 2\re \Big[ \sum_{\sigma \in \mathfrak A} \sum_{\vec k \in C_\sigma(k)} \frac1{(2\pi L)^4} \frac{e^{i\Delta(\vec k) t} -1 }{i\Delta (\vec k)}e^{-i\Delta(\vec k_\sigma) t} A_{\vec k} \overline{A_{\vec k_\sigma}} \E(g_{\vec k} \overline{g_{\vec k_\sigma}})\Big]
$$
and 
$$
B = 2\re \Big[ \sum_{\sigma \in \mathfrak B} \sum_{\vec k \in C_\sigma(k)} \frac1{(2\pi L)^4} \frac{e^{i\Delta(\vec k) t} -1 }{i\Delta (\vec k)}e^{-i\Delta(\vec k_\sigma) t} A_{\vec k} \overline{A_{\vec k_\sigma}} \E(g_{\vec k} \overline{g_{\vec k_\sigma}})\Big]
$$
where $\mathfrak B$ is the complementary of $\mathfrak A$ in $\mathfrak S_5$.

If $\sigma \in \mathfrak A$, then $\Delta (\vec k) = \Delta(\vec k_\sigma)$, $A_{\vec k} \overline{A_{\vec k_\sigma}} = \prod_{j=1}^5 |a(k_j)|^2$ and $\E(g_{\vec k} \overline{g_{\vec k_\sigma}}) = \E(|g_{\vec k}|^2)$. Therefore
$$
A =   \sum_{\sigma \in \mathfrak A} \sum_{\vec k \in C_\sigma(k)} \frac2{(2\pi L)^4} \frac{\sin(\Delta(\vec k) t)  }{\Delta (\vec k)} \prod_{j=1}^5 |a(k_j)|^2 \E(|g_{\vec k}|^2 ).
$$

Write 
$$
C_{\sigma,\neq} (k)  = \{\vec k \in C_\sigma (k) \; |\; \forall j_1\neq j_2 , k_{j_1} \neq k_{j_2}\}
$$
and $C_{\sigma,=}(k)$ its complementary in $C_\sigma(k)$. For all $\vec k \in C_{\sigma,\neq }(k)$ we have $\E(|g_k|^2) = 1$ and thus
$$
A = A_\neq + A_=
$$
with
$$
A_\neq   =   \sum_{\sigma \in \mathfrak A} \sum_{\vec k \in C_\sigma(k)} \frac2{(2\pi L)^4} \frac{\sin(\Delta(\vec k) t)  }{\Delta (\vec k)} \prod_{j=1}^5 |a(k_j)|^2
$$
and 
$$
A_= =   \sum_{\sigma \in \mathfrak A} \sum_{\vec k \in C_{\sigma,=}(k)} \frac2{(2\pi L)^4} \frac{\sin(\Delta(\vec k) t)  }{\Delta (\vec k)} \prod_{j=1}^5 |a(k_j)|^2( \E(|g_{\vec k}|^2-1) ).
$$

Since $|\E(|g_{\vec k}|^2 -1|\leq 5!$ and $|\Delta(\vec k)| \geq\nu^{-1}$, we get
$$
|A_=| \leq  5 ! \nu \sum_{\sigma \in \mathfrak A} \sum_{\vec k \in C_{\sigma,=}(k)} \frac2{(2\pi L)^4}  \prod_{j=1}^5 |a(k_j)|^2.
$$
Because satisfying ``$\exists j \neq l$ such that $k_j = k_l$" is a constraint independent from $\sum_{j=1}^{5} (-1)^j k_j +k=0$, we get that
$$
A_= = O_a(L^{-1} \nu).
$$

We estimate $B$, we recall that if $\sigma$ does not conserve parity and $\vec k \in C_\sigma(k)$ then $\E(g_{\vec k} \overline{g_{\vec k_\sigma}})\neq 0$ implies $\vec k \in C_{\sigma,=}(k)$. Therefore
$$
B = 2\re \Big[ \sum_{\sigma \in \mathfrak B} \sum_{\vec k \in C_{\sigma,=}(k)} \frac1{(2\pi L)^4} \frac{e^{i\Delta(\vec k) t} -1 }{i\Delta (\vec k)}e^{-i\Delta(\vec k_\sigma) t} A_{\vec k} \overline{A_{\vec k_\sigma}} \E(g_{\vec k} \overline{g_{\vec k_\sigma}})\Big]
$$
Since $|\E(g_{\vec k} \overline{g_{\vec k_\sigma}})| \leq 5!$ and since 
$$
\Big| \frac{e^{i\Delta(\vec k) t} -1 }{i\Delta (\vec k)}e^{-i\Delta(\vec k_\sigma) t}\Big| \leq 2 \nu
$$
we get
$$
|B| \leq 2\cdot 5!\nu \sum_{\sigma \in \mathfrak B} \sum_{\vec k \in C_{\sigma,=}(k)} \frac1{(2\pi L)^4}\prod_{j=1}^5 |a(k_j)|^2 .
$$
Because belonging to $C_{\sigma,=}(k)$ implies two independent linear constraint on $\vec k$, we get
$$
B = O_a(L^{-1}\nu).
$$
\end{proof}

\section{Final limits and proof of the result}\label{sec:limits}

We sum up quickly what we have done so far. Since for all $n>1$, $T_1,T_2 \in \mathcal T_n$,$\sigma\in \mathfrak S_{4n+1}$, $k\in \Z/L$ and $t\in \R$, we have thanks to Proposition \ref{prop:CaseB}
\[
V(T_1,T_2,\sigma,\varphi_1,\varphi_2,k,t) = O_{T_1,T_2,\sigma,\varphi_1,\varphi_2,a,\alpha}(t^{n-1}\nu^{(1+\alpha)n/2}L^{-(n-1)/2}) + O_{n,a}(t^{n-2}L^{4n-2}),
\]
we get
\[
\partial_t \E(|v_{n,L}(k,t)|^2) = O_{n,\alpha,a}(t^{n-1}\nu^{(1+\alpha)n/2}L^{-(n-1)/2}) + O_{n,a}(t^{n-2}L^{4n-2}),
\]

We deduce that for all $t\in \R$, we have that
\[
\varepsilon^{2(n-1)} \partial_t \E(|v_{n,L}(k,t\varepsilon^{-2})|^2) = O_{n,a,t,\alpha}(\nu^{(1+\alpha)n/2}L^{-(n-1)/2}) +O_{n,a,t}(\varepsilon^2L^{4n-2}).
\]
therefore
\[
\sum_{n=2}^N \varepsilon^{2(n-1)} \partial_t \E(|v_{n,L}(k,t\varepsilon^{-2})|^2) = O_{N,a,t}\Big( \frac{\nu^{1+\alpha}}{L^{1/2}}\Big) +O_{N,a,t}(\varepsilon^2 L^{4N-2}).
\]

We recall that thanks to Proposition \ref{prop:n=1}, we have 
$$
\partial_t \E(|v_1(t,k)|^2) = \sum_{\sigma \in \mathfrak A} \sum_{\vec k \in C_\sigma(k)} \frac{2}{(2\pi L)^4} \frac{\sin(\Delta(\vec k) t) }{\Delta (\vec k)} \prod_{j=1}^5|a(k_j)|^2 + O_a(L^{-1}\nu).
$$

We set
$$
I_{\varepsilon,L}(k,t) = \sum_{\sigma \in \mathfrak A} \sum_{\vec k \in C_\sigma(k)} \frac{2}{(2\pi L)^4} \frac{\sin(\Delta(\vec k) t\varepsilon^{-2}) }{\Delta (\vec k)} \prod_{j=1}^5|a(k_j)|^2 .
$$

We introduce test functions.

\begin{proposition} Let $f,g$ be two smooth, compactly supported functions on $\R$. Let $M\geq N\in \N^*$. We have 
\[
\partial_t \E\Big( \an{P_N U_L,f} \an{g,P_M U_L}\Big) = \frac1{2\pi L} \sum_{k\in Z/L} \hat f(k) \overline{\hat g(k)} I_{\varepsilon,L}(k,t) + O_{N,a,k,t,f,g}\Big( \frac{\nu^{1+\alpha}}{\sqrt L} + \varepsilon^2 L^{N-3}\Big).
\]
\end{proposition}

\begin{proof} Set $n_0 = \lfloor \frac{N-1}{4}\rfloor$. We have 
\[
P_N U_L (t) = \sum_{n=0}^{n_0} \varepsilon^n u_{n,L}(\varepsilon^{-2}t),
\]
hence
\[
\an{P_N U_L(t), f} = \sum_{n=0}^{n_0} \varepsilon^n \an{u_{n,L}(\varepsilon^{-2}t),f}.
\]
What is more,
\[
u_{n,L} (t\varepsilon^{-2}) = \sum_{k\in \Z/L} \hat u_{n,L} (t\varepsilon^{-2},k) \frac{e^{ikx}}{\sqrt{2\pi L}}
\]
and we recall that because the support of $a$ is compact, the sum over $k$ is finite. Hence
\[
\an{P_N U_L(t), f} = \frac1{\sqrt{2\pi L}}\sum_{k\in \Z/L} \sum_{n=0}^{n_0}\varepsilon^n \overline{\hat u_{n,L}(t\varepsilon^{-2},k)} \hat f(k).
\]
For the same reasons
\[
\an{g,P_M U_L(t)} = \frac1{\sqrt{2\pi L}}\sum_{k'\in \Z/L} \sum_{m=0}^{m_0}\varepsilon^{m} \hat u_{m,L}(t\varepsilon^{-2},k')\overline{ \hat g(k')}
\]
where $m_0 = \lfloor \frac{M-1}4\rfloor$.

Therefore,
\[
\E\Big( \an{P_N U_L,f} \an{g,P_M U_L}\Big) = \frac1{2\pi L} \sum_{k,k'} \hat f(k) \overline{\hat g(k')}\sum_{n,m} \varepsilon^{n+m}\E( \overline{\hat u_{n,L}(t\varepsilon^{-2},k)}\hat u_{m,L}(t\varepsilon^{-2},k')).
\]
We recall that
\[
\E( \overline{\hat u_{n,L}(t\varepsilon^{-2},k)}\hat u_{m,L}(t\varepsilon^{-2},k'))
\]
is null unless $n=m$ and $k=k'$. Hence,
\[
\E\Big( \an{P_N U_L,f} \an{g,P_M U_L}\Big) = \frac1{2\pi L} \sum_{k\in \Z/L} \hat f(k) \overline{\hat g(k)}\sum_{n=0}^{n_0}\varepsilon^{2n} \E( |\hat u_{n,L}(t\varepsilon^{-2},k)|^2).
\]
We have 
\[
\E( |\hat u_{n,L}(t\varepsilon^{-2},k)|^2) = \E( |v_{n,L}(t\varepsilon^{-2},k)|^2)
\]
hence
\[
\partial_t \E\Big( \an{P_N U_L,f} \an{g,P_M U_L}\Big) = \frac1{2\pi L} \sum_{k\in \Z/L} \hat f(k) \overline{\hat g(k)}\sum_{n=1}^{n_0} \varepsilon^{2(n-1)}\partial_t \E( |v_{n,L}(k)|^2)(t\varepsilon^{-2}).
\]
We recall
\[
\sum_{n=2}^{n_0} \varepsilon^{2(n-1)}\partial_t \E( |v_{n,L}(k)|^2)(t\varepsilon^{-2})= O_{n_0,a,t} \Big( \frac{\nu}{L^{1/2}} + \varepsilon^2L^{4n_0-2}\Big)
\]
hence  
\[
\frac1{2\pi L} \sum_{k\in \Z/L} \hat f(k) \overline{\hat g(k)}\sum_{n=2}^{n_0} \varepsilon^{2(n-1)}\partial_t \E( |v_{n,L}(k)|^2)(t\varepsilon^{-2}) = O_{n_0,a,t,f,g} \Big( \frac{\nu}{L^{1/2}} + \varepsilon^2L^{4n_0-2}\Big).
\]

Finally,
\[
\partial_t \E(|v_{1,L}(k)|^2) (t\varepsilon^{-2})= I_{\varepsilon,L}(k,t) + O_a(L^{-1}\nu)
\]
hence
\[
\frac1{2\pi L} \sum_{k\in \Z/L} \hat f(k) \overline{\hat g(k)}
\partial_t \E(|v_{1,L}(k)|^2) (t\varepsilon^{-2}) = \frac1{2\pi L} \sum_{k\in \Z/L} \hat f(k) \overline{\hat g(k)}I_{\varepsilon,L}(k,t) + O_{a,f,g}(L^{-1}\nu)
\]
which concludes the proof.
\end{proof}

We set 
\[
I_L(t) = \frac1{2\pi L} \sum_{k\in \Z/L} \hat f(k) \overline{\hat g(k)}I_{\varepsilon,L}(k,t).
\]

\subsection{Case \texorpdfstring{$t$}{t} is dyadic}

\begin{proposition}\label{prop:sumtointegral} Assume that $\varepsilon^{-2} =2\pi 2^{L} L^2 + \rho(L)$ and that $t$ is dyadic. We have 
$$
I_{L}(t) = \frac3{4\pi^5} \int_{\vec k \in B }   \frac{\sin(\Delta(\vec k) t\rho }{\Delta (\vec k)} \prod_{j=1}^5|a(k_j)|^2 \delta(k + \sum_{j=1}^5 (-1)^j k_j) \hat f(k) \overline{\hat g(k)}d\vec k + O_a(\frac{\rho^2}{L} + \frac{\rho\mu}\nu + \frac\rho\mu)
$$
where $B = \{(k,k_1,k_2,k_3,k_4,k_5) \in \R^6 \; |\; |k-k_1+k_2-k_3|\geq \frac1{\mu}\}$, where
\[
\Delta(\vec k) = k^2 + \sum_{j=1}^5 (-1)^j k_j^2
\]
and $d\vec k = dk \prod_{j=1}^5 dk_j$.
\end{proposition}

\begin{proof} Let $\vec k \in C_\sigma : = \{ (k,\vec k') |\vec k' \in C_\sigma(k)\}$, we have 
$$
\Delta(\vec k) t\varepsilon^{-2} = \Delta(\vec k) L^2 t 2\pi 2^L + \Delta(\vec k) t \rho.
$$
Therefore, since $t$ is dyadic, above a certain rank
$$
\Delta(\vec k) t\varepsilon^{-2} \in \Delta(\vec k) t \rho + 2 \pi \Z.
$$
Hence 
$$
I_{\varepsilon,L}(k,t) = \sum_{\sigma \in \mathfrak A} \sum_{\vec k \in C_\sigma(k)} \frac{2}{(2\pi L)^4} \frac{\sin(\Delta(\vec k) t\rho )}{\Delta (\vec k)} \prod_{j=1}^5|a(k_j)|^2 
$$
Since 
$$
\Big| \frac{\sin(\Delta(\vec k) t\rho )}{\Delta (\vec k)}  \Big| \leq t\rho
$$
and
$$
\Big| \nabla \Big( \frac{\sin(\Delta(\vec k) t\rho }{\Delta (\vec k)}\Big) \Big| \leq |t \rho|^2 \; |\nabla \Delta| \big\| \frac{x\cos x -\sin x}{x^2} \big\|_{L^\infty}
$$
we get the convergence towards the Riemann integral since $\rho^2 = o(\nu^2) = o(L)$,
$$
I_{L}(t) = \sum_{\sigma \in \mathfrak A} \int_{ B_\sigma} \frac{2}{(2\pi )^5} \frac{\sin(\Delta(\vec k) t\rho }{\Delta (\vec k)} \prod_{j=1}^5|a(k_j)|^2 \delta(k + \sum_{j=1}^5 (-1)^j k_j)\hat f(k) \overline{\hat g(k)}d\vec k + O_{a,f,g,t}(\rho^2 L^{-1})
$$
where
$$
B_\sigma  = \{(k,k_1,k_2,k_3,k_4,k_5)\; |\; |k-k_1+k_2-k_3|\geq \mu^{-1}\; |k-k_{\sigma(1)}+k_{\sigma(2)}-k_{\sigma(3)}|\geq \mu^{-1}\; \Delta(\vec k) \geq \nu^{-1}\}.
$$
We have 
\begin{multline*}
\int_{|\Delta(\vec k)|< \frac1{\nu},|k-k_1+k_2-k_3|\geq \frac1{\mu}}  \prod_{j=1}^5|a(k_j)|^2 \delta(k + \sum_{j=1}^5 (-1)^j k_j)|\hat f(k) \hat g(k)|d\vec k = \\
\int_{|k-k_1+k_2-k_3|\geq \frac1{\mu}}dk dk_1 dk_2 dk_3 |a(k_1)|^2 |a(k_2)|^2 |a(k_3)|^2 |\hat f(k) \hat g(k)| \\
\int_{|\Delta(\vec k)|< \frac1{\nu}} dk_4 |a(k_4)|^2 |a(k + \sum_{j=1}^4 (-1)^j k_j)|^2 .
\end{multline*}
Since
$$
\Delta(\vec k) = k^2 - k_1^2 + k_2^2 - k_3^4 - (k +k + \sum_{j=1}^3 (-1)^j k_j)^2 -2 k_4 (k + \sum_{j=1}^3 (-1)^j k_j)
$$
we get that $|\Delta (\vec k)|\leq \frac1{\nu}$ implies 
$$
|k_4 - \frac{k^2 - k_1^2 + k_2^2 - k_3^4 - (k +k + \sum_{j=1}^3 (-1)^j k_j)^2}{2(k + \sum_{j=1}^3 (-1)^j k_j) }|\leq \frac{\mu}{\nu}
$$
hence
\begin{multline*}
\int_{|\Delta(\vec k)|< \frac1{\nu},|k-k_1+k_2-k_3|\geq \frac1{\mu}}  \prod_{j=1}^5|a(k_j)|^2 \delta(k + \sum_{j=1}^5 (-1)^j k_j) |\hat f(k) \hat g(k)|d\vec k\\
\leq \sup |a|^4 \frac\mu\nu \int dk dk_1 dk_2 dk_3 |a(k_1)|^2 |a(k_2)|^2 |a(k_3)|^3|\hat f(k) \hat g(k)|.
\end{multline*}
Thus, we get
\begin{multline*}
I_{L}(t) = \sum_{\sigma \in \mathfrak A} \int_{ B'_\sigma} \frac{2}{(2\pi )^5} \frac{\sin(\Delta(\vec k) t\rho }{\Delta (\vec k)} \prod_{j=1}^5|a(k_j)|^2\hat f(k) \overline{\hat g(k)} \delta(k + \sum_{j=1}^5 (-1)^j k_j) \prod_{j=1}^5 dk_j +\\
 O_{a,t,f,g}(\rho^2 L^{-1}) + O_{a,t,f,g}(\frac{\rho\mu}{\nu})
\end{multline*}
where 
$$
B'_\sigma  = \{(k,k_1,k_2,k_3,k_4,k_5)\; |\; |k-k_1+k_2-k_3|\geq \frac1{\mu}\; |k-k_{\sigma(1)}+k_{\sigma(2)}-k_{\sigma(3)}|\geq \frac1{\mu}\}.
$$

We have
\begin{multline*}
\int_{|k + \sum_{j=1}^3 (-1)^jk_{\sigma (j)}|< \frac1{\mu}} |\hat f(k) \hat g(k)| \prod_{j=1}^5|a(k_j)|^2 \delta(k + \sum_{j=1}^5 (-1)^j k_j) d\vec k =\\
 \int dk dk_1 dk_2 dk_4 |a(k_1)|^2 |a(k_2)|^2 |a(k_4)|^2 |\hat f (k) \hat g(k)|  \\
 \int_{|k + \sum_{j=1}^3 (-1)^jk_{\sigma (j)}|< \frac1{\mu}} dk_3 |a(k_3)|^2 |a(k + \sum_{j=1}^4 (-1)^j k_j)|^2 .
\end{multline*}
We get
\begin{multline*}
\int_{|k + \sum_{j=1}^3 (-1)^jk_{\sigma (j)}|< \frac1{\mu}}  \prod_{j=1}^5|a(k_j)|^2 \delta(k + \sum_{j=1}^5 (-1)^j k_j) |\hat f(k) \hat g(k)| d\vec k\\
\leq \int dk dk_1 dk_2 dk_4 |a(k_1)|^2 |a(k_2)|^2 |a(k_4)|^2 |\hat f (k) \hat g(k)|  \sup |a|^4 \frac1{\mu}.
\end{multline*}

Thus, we get
$$
I_{L}(t) = \sum_{\sigma \in \mathfrak A} \int_{ B} \frac{2}{(2\pi L)^5} \frac{\sin(\Delta(\vec k) t\rho }{\Delta (\vec k)} \prod_{j=1}^5|a(k_j)|^2\hat f(k) \overline{\hat g(k)} \delta(k + \sum_{j=1}^5 (-1)^j k_j)d\vec k + O_{a,t,f,g}(\rho^2L^{-1}+ \frac{\rho\mu}\nu +\frac\rho\mu)
$$
which concludes the proof since the cardinal of $\mathfrak A$ is $3! 2! = 12$.
\end{proof}

Let 
$$
J_L (t)= \frac3{4\pi^5}\int_{\vec k \in B }   \frac{\sin(\Delta(\vec k) t\rho }{\Delta (\vec k)} \prod_{j=1}^5|a(k_j)|^2 \hat f(k) \overline{\hat g(k)} \delta(k + \sum_{j=1}^5 (-1)^j k_j) d\vec k.
$$

\begin{proposition}\label{prop:diracdeltas} We have that
$$
\lim_{L \rightarrow \infty} J_L(t) = \frac3{4\pi^4}\int_{\R^6} \delta(k + \sum_{j=1}^5 (-1)^j k_j)  \delta (\Delta (\vec k) ) \frac1{k-k_1 + k_2 - k_3 } \prod_{j=1}^5|a(k_j)|^2\hat f(k) \overline{\hat g(k)} d\vec k
$$
and besides the integral converges.
\end{proposition}

\begin{proof}

We have 
$$
J_L(t) = \frac3{4\pi^5} \int_{\vec k \in C} \frac{\sin(D(\vec k)t\rho)}{D(\vec k)} \prod_{j=1}^4|a(k_j)|^2 |a(k+ \sum_{j=1}^4 (-1)^j k_j)|^2\hat f(k) \overline{\hat g(k)} d\vec k
$$
where 
$$
C = \{(k,k_1,\hdots,k_4)|\; |k-k_1+k_2-k_3|\geq \frac1{\mu}\}
$$
and
$$
D(k,k_1,\hdots, k_4) = \Delta(k,k_1,\hdots, k_4, k-\sum_{j=1}^4(-1)^j k_j).
$$
To lighten the notations, we write equally
$$
C = \{(k_1,k_2,k_3)|\; |k-k_1+k_2-k_3|\geq \frac1{\mu}\}.
$$
We get
$$
J_L(t) = \frac3{4\pi^5} \int_{C} dkdk_1dk_2dk_3 \hat f(k) \overline{\hat g(k)} |a(k_1)|^2|a(k_2)|^2|a(k_3)|^2 J_L^4 (t,k,k_1,k_2,k_3)
$$
where
$$
J_L^4(t,k,k_1,k_2,k_3) = \int dk_4  \frac{\sin(D(\vec k)t\rho)}{D(\vec k)} |a(k_4)|^2|a(k+ \sum_{j=1}^4 (-1)^j k_j)|^2.
$$
Given a fixed $k,k_1,k_2,k_3 \in C$, the derivative of $k_4\mapsto D(\vec k)$ being
$$
-2 (k-k_1+k_2-k_3)
$$
and denoting $\bar k =  k-k_1+k_2-k_3\neq 0$, we get by integration by parts
$$
J_L^4(t,k,k_1,k_2,k_3) = J_{L,1}^4 (t,k,k_1,k_2,k_3) + J_{L,2}^4 (t,k,k_1,k_2,k_3)
$$
with
$$
J_{L,1}^4 (t,k,k_1,k_2,k_3) = \int dk_4 \frac{1-\cos(D(\vec k)t\rho)}{D^2(\vec k) t\rho} |a(k_4)|^2 |a(\bar k + k_4)|^2
$$
and 
\begin{multline*}
J_{L,2}^4 (t,k,k_1,k_2,k_3) =\\
 \int dk_4 \frac{1-\cos(D(\vec k)t\rho)}{\bar k D(\vec k) t\rho} \re \left( a'(k_4) \bar a (k_4) |a(\bar k + k_4)|^2+|a(k_4)|^2 a'(\bar k + k_4) \bar a (\bar k +k_4)\right).
\end{multline*}
For $j=1,2$, write
$$
J_{L,j}(t) =  \frac3{4\pi^5} \int_{C} dkdk_1dk_2dk_3 \hat f(k) \overline{\hat g(k)} |a(k_1)|^2|a(k_2)|^2|a(k_3)|^2 J_{L,j}^4 (t,k_1,k_2,k_3).
$$

We estimate $J_{L,2}(t)$. First, we have 
\[
\frac{1-\cos(D(\vec k)t\rho)}{D(\vec k)t\rho} \lesssim \frac1{(D(\vec k)t\rho)^{1/2}}
\]
and thus
$$
J_{L,2}^4 (t,k,k_1,k_2,k_3) \leq \|a'\|_{L^\infty} \|a\|_{L^\infty}^2 \frac1{|\bar k|\sqrt{t\rho} }\int dk_4 \frac{|a(k_4)|}{\sqrt{D(\vec k)}}
$$
which implies in turn
$$
J_{L,2}(t) \lesssim_a \frac1{\sqrt{t\rho}} \int_C dk dk_1dk_2dk_3dk_4 |\hat f (k) \hat g(k)|\frac{|a(k_4)|\prod_{j=1}^3|a(k_j)|^2}{|\bar k|\sqrt{D(\vec k)}}.
$$

Then, we see that 
$$
D(\vec k) = \bar D(k_1,k_2,k_3) -2\bar k k_4
$$
with 
$$
\bar D (k_1,k_2,k_3) = k^2 -k_1^2+k_2^2-k_3^2 -\bar k^2 = \tilde D(k_1,k_2) + 2 \bar k k_3
$$
with
$$
\tilde D(k_1,k_2) = k^2 -k_1^2+k_2^2 - (k-k_1+k_2)^2 = -2(k-k_1)(k_2-k_1).
$$
We divide the domain of integration of $J_{L,2}$ in three parts as
$$
J_{L,2}(t) = J_{L,3}(t) + J_{L,4}(t) + J_{L,5}(t)
$$
with
$$
J_{L,3}(t) =  \frac3{4\pi^5} \int_{C\cap\{\bar D \leq \tilde D/2\}} dkdk_1dk_2dk_3 |\hat f(k)\hat g(k)||a(k_1)|^2|a(k_2)|^2|a(k_3)|^2 J_{L,2}^4 (t,k_1,k_2,k_3),
$$

$$
J_{L,4}(t,k) =  \frac3{4\pi^5} \int_{C\cap\{\bar D > \tilde D/2, D\leq \bar D/2\}} dkdk_1dk_2dk_3  |\hat f(k)\hat g(k)||a(k_1)|^2|a(k_2)|^2|a(k_3)|^2 J_{L,2}^4 (t,k_1,k_2,k_3)
$$
and
$$
J_{L,5}(t,k) =  \frac3{4\pi^5} \int_{C\cap\{\bar D > \tilde D/2, D> \bar D/2\}} dkdk_1dk_2dk_3  |\hat f(k)\hat g(k)||a(k_1)|^2|a(k_2)|^2|a(k_3)|^2 J_{L,2}^4 (t,k_1,k_2,k_3)
$$

For $J_{L,5}$, we use that $D> \tilde D/4$ to get
$$
|J_{L,5}|\lesssim_a \frac1{\sqrt{t\rho}} \int_C dk dk_1dk_2dk_3dk_4  |\hat f(k)\hat g(k)|\frac{|a(k_4)|\prod_{j=1}^3|a(k_j)|^2}{|\bar k|\sqrt{\tilde D(k_1,k_2)}}.
$$
We use that we integrate over $C$ to get
$$
|J_{L,5}|\lesssim_a \frac{\ln^2 \mu}{\sqrt{t\rho}} \int dkdk_1dk_2dk_3dk_4  |\hat f(k)\hat g(k)|\frac{|a(k_4)|\prod_{j=1}^3|a(k_j)|^2}{|\bar k|\ln^2(|\bar k|)\sqrt{(k-k_1)(k_2-k_1)}}.
$$
We recall that $\bar k = k-k_1 + k_2 - k_3$, therefore, by integrating first in $k_4$ then in $k_3$ then in $k_2$ then in $k_1$ and finally in $k$, as in
$$
|J_{L,5}|\lesssim_a \frac{\ln^2 \mu}{\sqrt{t\rho}} \int dk |\hat f(k)\hat g(k)| \int dk_1 \frac{|a(k_1)|^1}{\sqrt{|k-k_1|}} \int dk_2 \frac{|a(k_2)|^2}{\sqrt{|k_2-k_1|}} \int dk_3 \frac{|a(k_3)|^3}{|\bar k|\ln^2(|\bar k|)} \int dk_4 |a(k_4)|
$$
we get
$$
|J_{L,5}|\lesssim_{a,f,g}  \frac{\ln^2 \mu}{\sqrt{t\rho}}
$$
which goes to $0$ as $L$ goes to $\infty$ as $\ln^2\mu = o(\sqrt \rho)$.

For $J_{L,4}$ we use that since $D\leq \bar D/2$, we have that
$$
|k_4| = \frac{|\bar D - D|}{2|\bar k|} \geq \frac{|\bar D|}{4 |\bar k|}
$$
and therefore since $|\bar D| > |\tilde D|/2$, we get
$$
|k_4|  \geq \frac{|\tilde D|}{8|\bar k|} \textrm{ and } |a(k_4)| \leq \frac1{\sqrt{|k_4|}\an {k_4}^2} \|\sqrt{|x|}\an{x}^2 a\|_{L^\infty} \lesssim_a \frac{\sqrt{|\bar k|}}{\sqrt{\tilde D} \an{k_4}^2}
$$
from which we get
\begin{multline*}
|J_{L,4}(t,k)| 
\lesssim_a \frac{\ln^2 \mu}{\sqrt{t\rho}}\\
 \int dk  |\hat f(k)\hat g(k)| \int dk_1 \frac{|a(k_1)|^2}{\sqrt{|k-k_1|}} \int dk_2 \frac{|a(k_2)|^2}{\sqrt{|k_2-k_1|}} \int dk_3 \frac{|a(k_3)|^3}{|\bar k|\ln^2(|\bar k|)} \int dk_4 \an{k_4}^{-2} \frac{\sqrt{|\bar k|}}{\sqrt{D(\vec k)}}.
\end{multline*}
Since $\frac{D(\vec k)}{|\bar k|} = -2k_4 + \alpha(k_1,k_2,k_3)$ where $\alpha$ is a map depending only on $k_1,k_2,k_3$ we get
$$
|J_{L,4}(t,k)| \lesssim_{a,f,g} \frac{\ln^2 \mu}{\sqrt{t\rho}} .
$$

For $J_{L,3}$ we use that since $|\bar D| \leq |\tilde D|/2$ we have that
$$
|k_3| = \frac{|\tilde D - \bar D|}{2|\bar k|} \geq \frac{|\tilde D|}{4|\bar k|}
$$
and therefore,
$$
|a(k_3)|\lesssim_a \sqrt{\frac{|\bar k|}{\tilde D}}.
$$
We get as previously
\begin{multline*}
|J_{L,3}(t,k)| \lesssim_a \frac{\ln^2 \mu}{\sqrt{t\rho}} \\
\int dk  |\hat f(k)\hat g(k)|\int dk_1 \frac{|a(k_1)|^2}{\sqrt{|k-k_1|}} \int dk_2 \frac{|a(k_2)|^2}{\sqrt{|k_2-k_1|}} \int dk_3 \frac{|a(k_3)|}{|\bar k|\ln^2(|\bar k|)} \int dk_4 |a(k_4)| \frac{\sqrt{|\bar k|}}{\sqrt{D(\vec k)}}
\end{multline*}
and we conclude as for $J_{L,4}$. 

We now compute the limit of $J_{L,1}$. By the change of variable $\xi = D(\vec k)t\rho$, we have
$$
J_{L,1}^4 (t,k,k_1,k_2,k_3) = \int \frac{d\xi}{2|\bar k|} \frac{1-\cos(\xi)}{\xi^2} |a(\frac{\bar D -\xi/(t\rho)}{2\bar k})|^2 |a(\bar k + \frac{\bar D -\xi/(t\rho)}{2\bar k})|^2.
$$ 
We divide the integral into two parts as
$$
J_{L,1}^4(t,k,k_1,k_2,k_3)=K_{L,1}^4(t,k,k_1,k_2,k_3) + K_{L,2}^4(t,k,k_1,k_2,k_3)
$$
with
$$
K_{L,1}^4(t,k,k_1,k_2,k_3) = \int_{|\xi|\leq |\bar D| t\rho/2} \frac{d\xi}{2|\bar k|} \frac{1-\cos(\xi)}{\xi^2} |a(\frac{\bar D -\xi/(t\rho)}{2\bar k})|^2 |a(\bar k + \frac{\bar D -\xi/(t\rho)}{2\bar k})|^2
$$
and 
$$
K_{L,2}^4(t,k,k_1,k_2,k_3) = \int_{|\xi|> |\bar D| t\rho/2} \frac{d\xi}{2|\bar k|} \frac{1-\cos(\xi)}{\xi^2} |a(\frac{\bar D -\xi/(t\rho)}{2\bar k})|^2 |a(\bar k + \frac{\bar D -\xi/(t\rho)}{2\bar k})|^2.
$$
We denote for $j=1,2$,
$$
K_{L,j}(t,k) = \frac3{4\pi^5} \int_{C} dkdk_1dk_2dk_3\hat f (k) \overline{\hat g(k)} |a(k_1)|^2|a(k_2)|^2|a(k_3)|^2 K_{L,j}^4 (t,k,k_1,k_2,k_3).
$$
For $K_{L,2}$, using the bound on $|\xi|$ we get
$$
|K_{L,2}^4(t,k,k_1,k_2,k_3)| \lesssim_a \frac1{|\bar k|(t\rho)^{1/4} |\bar D|^{1/4}}\int d\xi \frac{1-\cos(\xi)}{|\xi|^{7/4}} \lesssim_a  \frac1{|\bar k|(t\rho)^{1/4}  |\bar D|^{1/4}}.
$$
We divide the integral on $k_1,k_2,k_3$ in two as previously as
$$
|K_{L,2}(t)| \lesssim_a K_{L,3}(t)  + K_{L,4}(t)
$$
with
$$
K_{L,3}(t) = \int_{\bar D \leq \tilde D/2 \cap C}  \frac1{|\bar k|(t\rho)^{1/4} |\bar D|^{1/4}} |\hat f(k) \hat g(k)| |a(k_1)|^2|a(k_2)|^2|a(k_3)|^2 dk_1dk_2dk_3
$$
and 
$$
K_{L,4}(t) = \int_{\bar D > \tilde D/2\cap C}  \frac1{|\bar k|(t\rho)^{1/4} |\bar D|^{1/4}} |\hat f(k) \hat g(k)||a(k_1)|^2|a(k_2)|^2|a(k_3)|^2dkdk_1dk_2dk_3.
$$

For $K_{L,4}$ we use the inequalities on $\bar k$ and $\bar D$ to get
$$
K_{L,4}(t) \lesssim \frac{\ln^2(\mu)}{(t\rho)^{1/4} }\int  \frac1{|\bar k|\ln^2(|\bar k|) |k-k_1|^{1/4} |k_1-k_2|^{1/4}} |\hat f(k) \hat g(k)||a(k_1)|^2|a(k_2)|^2|a(k_3)|^2dkdk_1dk_2dk_3
$$
which is integrable as previously, hence
$$
K_{L,4}(t,k) \lesssim_{a,f,g} \frac{\ln^2(\mu)}{(t\rho)^{1/4}}.
$$

For $K_{L,3}$ we recall that the inequality on $\bar D$ implies that
$$
|k_3| \geq \frac{\tilde D}{|\bar k|}.
$$
We use that
$$
|a(k_3)| \lesssim_a \frac{|\bar k|^{3/4}}{|\tilde D|^{3/4}}.
$$
We deduce
$$
K_{L,3} \lesssim_a \frac1{(t\rho)^{1/4} } \int \frac1{|\bar D|^{1/4}|\tilde D|^{3/4} |\bar k|^{1/4} }|\hat f(k) \hat g(k)||a(k_1)|^2|a(k_2)|^2|a(k_3)|dkdk_1dk_2dk_3.
$$
We have that
$$
|\bar D| \geq 2 |k_3 - \beta(k_1,k_2)|\; |k_3-\gamma(k_1,k_2)|
$$
where $\beta$ and $\gamma$ are two maps depending only on $k_1,k_2$. By using H\"older's inequality on $k_3$, we get
\begin{multline*}
K_{L,3} \lesssim_a  \frac1{(t\rho)^{1/4}}\int dk |\hat f(k) \hat g(k)|\\
\int dk_1  \frac{|a(k_1)|^2}{|k-k_1|^{3/4}} \int dk_2 \frac{ |a(k_2)|^2 }{|k_2-k_1|^{3/4}}\|\frac{|a|^{1/3}}{|k_3-\beta|^{1/4}}\|_{L^3(k_3)}\|\frac{|a|^{1/3}}{|k_3-\gamma|^{1/4}}\|_{L^3(k_3)}\|\frac{ |a|^{1/3}}{|\bar k|^{1/4}}\|_{L^3(k_3)}.
\end{multline*}
Since $\frac34 < 1$, we get
$$
K_{L,3} \lesssim_{f,g} \frac1{(t\rho)^{1/4}}.
$$

We now turn to $K_{L,1}$. We set
\begin{multline*}
f_L(k,k_1,k_2,k_3,\xi) = 
{\bf 1}_{|\bar k|\geq \frac1{\mu}}{\bf 1}_{|\xi|\leq |\bar D| t\rho/2}\\
 \frac1{2|\bar k|} \frac{1-\cos(\xi)}{\xi^2} |a(\frac{\bar D -\xi/(t\rho)}{2\bar k})|^2 |a(\bar k + \frac{\bar D -\xi/(t\rho)}{2\bar k})|^2 \hat f(k) \overline{\hat g(k)}|a(k_1)|^2|a(k_2)|^2|a(k_3)|^2.
\end{multline*}
When $L$ goes to $\infty$ we have that $\rho \rightarrow\infty$ and thus $f_L$ converges almost surely to
$$
f_\infty(k,k_1,k_2,k_3,\xi) = \frac1{2|\bar k|} \frac{1-\cos(\xi)}{\xi^2} |a(\frac{\bar D }{2\bar k})|^2 |a(\bar k + \frac{\bar D }{2\bar k})|^2\hat f(k) \overline{\hat g(k)} |a(k_1)|^2|a(k_2)|^2|a(k_3)|^2.
$$

What is more, because $|\xi|\leq |\bar D| t\rho/2$, we have that
$$
\Big|\frac{\bar D -\xi/(t\rho)}{2\bar k}\Big| \geq \Big| \frac{\bar D}{4\bar k}\Big|
$$
from which we deduce that
$$
|a(\frac{\bar D -\xi/(t\rho)}{2\bar k})|^2 \lesssim_a \frac{|\bar k|^{1/2}}{|\bar D|^{1/2}}.
$$
Therefore, if $|\bar D|\geq |\tilde D|/2$, we get 
$$
|f_L(k,k_1,k_2,k_3,\xi)| \lesssim_a \frac1{|\bar k|^{1/2}|\tilde D|^{1/2}} \frac{1-\cos(\xi)}{\xi^2} |\hat f(k)\hat g(k)| |a(k_1)|^2|a(k_2)|^2|a(k_3)|^2.
$$
On the other hand, if $|\bar D|< |\tilde D|/2$ then 
$$
|k_3| > \frac{|\tilde D|}{4|\bar k|}
$$
and we get
$$
|f_L(k,k_1,k_2,k_3,\xi)| \lesssim_a \frac1{|\bar k|^{1/2}|\tilde D|^{1/2}} \frac{1-\cos(\xi)}{\xi^2} |\hat f(k)\hat g(k)| |a(k_1)|^2|a(k_2)|^2|a(k_3)|.
$$
Hence, for $k,k_1,k_2,k_3,k_4$ we have 
$$
|f_L(k_1,k_2,k_3,\xi)| \lesssim_a \frac1{|\bar k|^{1/2}|\tilde D|^{1/2}} \frac{1-\cos(\xi)}{\xi^2} |\hat f(k)\hat g(k)| |a(k_1)|^2|a(k_2)|^2|a(k_3)|,
$$
the map on the right hand side being integrable we can apply DCT and get that
\begin{multline*}
\lim_{L\rightarrow \infty} K_{L,1}(t) = \int f_\infty(k,k_1,k_2,k_3,\xi) \\
= \frac3{4\pi^4} \int dk_1dk_2dk_3 \frac1{|\bar k|} |a(\frac{\bar D }{2\bar k})|^2 |a(\bar k + \frac{\bar D }{2\bar k})|^2 |a(k_1)|^2|a(k_2)|^2|a(k_3)|^2
\end{multline*}
and the map below the integral is integrable.
\end{proof}

\subsection{Case \texorpdfstring{$t=\frac13$}{t is one third}}

\begin{proposition}\label{prop:sumtointegralonethird} Assume that $\varepsilon^{-2} =2\pi 2^{L} L^2 + \rho(L)$ and that $t = \frac13$. We have 
$$
I_{L}(t) = \frac1{12\pi^5} \int_{\vec k \in B }   \frac{\sin(\Delta(\vec k) t\rho }{\Delta (\vec k)} \prod_{j=1}^5|a(k_j)|^2 \delta(k + \sum_{j=1}^5 (-1)^j k_j) \hat f(k) \overline{\hat g(k)}d\vec k + O_a(\frac{\mu\nu}{L} + \frac{\rho^2}{L} + \frac{\rho\mu}\nu + \frac\rho\mu)
$$
where $B = \{(k,k_1,k_2,k_3,k_4,k_5) \in \R^6 \; |\; |k-k_1+k_2-k_3|\geq \frac1{\mu}\}$, where
\[
\Delta(\vec k) = k^2 + \sum_{j=1}^5 (-1)^j k_j^2
\]
and $d\vec k = dk \prod_{j=1}^5 dk_j$.
\end{proposition}

\begin{proof} First, we see that
\[
t = \frac13 = \sum_{n\geq 1} \frac1{2^{2n}}
\]
and therefore
\[
2^L t \in \frac{x_L}3 + \N
\]
where $x_L= 1$ if $L$ is even and $x_L = 2$ if $L$ is odd. We deduce that
\[
\Delta(\vec k) \varepsilon^{-2} t \in \Delta(\vec k) t\rho + \frac{2\pi x_L}{3} \Big[ (kL)^2 + \sum_{j=1}^5 (-1)^j (Lk_j)^2\Big] + 2\pi \Z.
\]
And therefore, we get
\begin{multline*}
\sin(\Delta(\vec k) \varepsilon^{-2}t) = \\
\sin(\Delta(\vec k) \rho t) \cos\Big( \frac{2\pi x_L}{3} \Big[ (kL)^2 + \sum_{j=1}^5 (-1)^j (Lk_j)^2\Big]\Big) + \cos(\Delta(\vec k) \rho t ) \sin \Big( \frac{2\pi x_L}{3} \Big[ (kL)^2 + \sum_{j=1}^5 (-1)^j (Lk_j)^2\Big]\Big).
\end{multline*}
To know the values of 
\[
\cos\Big( \frac{2\pi x_L}{3} \Big[ (kL)^2 + \sum_{j=1}^5 (-1)^j (Lk_j)^2\Big]\Big)
\]
and 
\[
\sin\Big( \frac{2\pi x_L}{3} \Big[ (kL)^2 + \sum_{j=1}^5 (-1)^j (Lk_j)^2\Big]\Big)
\]
it is sufficient to know the congruence of $Lk, Lk_j$ modulo $3$. Therefore, we write 
\[
I_L(t) =\frac2{(2\pi)^5} \sum_{\kappa} I_{\kappa,L}(t)
\]
with $\kappa$ a map from $[0,4]\cap \N$ to $\{-1,0,1\}$ and 
\[
I_{\kappa,L}(t) = \sum_{\sigma \in \mathfrak A} \frac1{L^5}\sum_{\vec k \in C_{\sigma,\kappa}} 
\frac{\sin(\Delta(\vec k) t \rho) b_{\kappa,L} + \cos(\Delta(\vec k) t\rho) c_{\kappa,L}}{\Delta(\vec k)} \prod_{j=1}^5 |a(k_j)|^2 \hat f(k) \overline{\hat g(k)}
\]
where 
\[
C_{\sigma,\kappa} = \{(k,k_1,k_2,k_3,k_4,k_5) \in C_\sigma \; |\; Lk \in \kappa(0) + 3\Z, \quad \forall j, \quad Lk_j \in \kappa(j) + 3\Z\}
\]
and 
\[
b_{\kappa,L}= \cos\Big( \frac{2\pi x_L}{3}  \sum_{j=0}^5 (-1)^j \kappa(j)^2\Big)
\]
and
\[
c_{\kappa,L} = \sin\Big( \frac{2\pi x_L}{3}  \sum_{j=0}^5 (-1)^j \kappa(j)^2\Big).
\]

We divide $I_{\kappa,L}$ in two as
\[
I_{\kappa,L} = J_{\kappa,L} + K_{\kappa,L}
\]
with
\[
J_{\kappa,L} = b_{\kappa,L} \sum_{\sigma \in \mathfrak A} \frac1{L^5}\sum_{\vec k \in C_{\sigma,\kappa}} 
\frac{\sin(\Delta(\vec k) t \rho)}{\Delta(\vec k)} \prod_{j=1}^5 |a(k_j)|^2 \hat f(k) \overline{\hat g(k)}
\]
and
\[
K_{\kappa,L} =  c_{\kappa,L}\sum_{\sigma \in \mathfrak A} \frac1{L^5}\sum_{\vec k \in C_{\sigma,\kappa}} 
\frac{ \cos(\Delta(\vec k) t\rho)}{\Delta(\vec k)} \prod_{j=1}^5 |a(k_j)|^2 \hat f(k) \overline{\hat g(k)}.
\]

With the same strategy as in the proof of Proposition \ref{prop:sumtointegral}, we have 
\[
J_{\kappa,L} = b_{\kappa,L} \frac{12}{3^5} \int_{\vec k \in B} \frac{\sin(\Delta(\vec k) t \rho)}{\Delta(\vec k)} \prod_{j=1}^5 |a(k_j)|^2 \hat f(k) \overline{\hat g(k)} d\vec k + O_{a,f,g}\Big( \frac{\rho^2}{L} + \frac{\rho\mu}\nu + \frac\rho\mu\Big).
\]
According to the program set in Appendix \ref{app:program}, $b_{\kappa,L} = 1$ in $99$ cases and is equal to 
\[
\cos(\frac{2\pi}{3}) = \cos(-\frac{2\pi}3) = -\frac12
\]
in $144$ cases. Therefore,
\[
\sum_\kappa J_{\kappa,L} = 12 \frac{99-72}{3^5} \int_{\vec k \in B} \frac{\sin(\Delta(\vec k) t \rho)}{\Delta(\vec k)} \prod_{j=1}^5 |a(k_j)|^2 \hat f(k) \overline{\hat g(k)} d\vec k + O_{a,f,g}\Big( \frac{\rho^2}{L} + \frac{\rho\mu}\nu + \frac\rho\mu\Big).
\]

Still according to the program in Appendix \ref{app:program}, $c_{\kappa,L} = 0$ in $99$ cases, is equal to $\frac{\sqrt 3}{2}$ in $72$ cases and is equal to $-\frac{\sqrt 3}2$ in $72$ cases. We fix an involution $\kappa \mapsto \bar \kappa$ such that $c_{\bar \kappa,L} = -c_{\kappa,L}$. We get
\begin{multline}\label{eqKKappa}
K_{\kappa,L} + K_{\bar \kappa,L} = c_{\kappa,L}\sum_{\sigma \in \mathfrak A} \frac1{L^5}\sum_{\vec k \in C_{\sigma,\kappa}} 
\frac{ \cos(\Delta(\vec k) t\rho) -1}{\Delta(\vec k)} \prod_{j=1}^5 |a(k_j)|^2 \hat f(k) \overline{\hat g(k)}\\
-c_{\kappa,L}\sum_{\sigma \in \mathfrak A} \frac1{L^5}\sum_{\vec k \in C_{\sigma,\bar \kappa}} 
\frac{ \cos(\Delta(\vec k) t\rho)-1}{\Delta(\vec k)} \prod_{j=1}^5 |a(k_j)|^2 \hat f(k) \overline{\hat g(k)}\\
+ c_{\kappa,L}\sum_{\sigma \in \mathfrak A} \frac1{L^5}\sum_{\vec k \in C_{\sigma,\kappa}} 
\frac1{\Delta(\vec k)} \prod_{j=1}^5 |a(k_j)|^2 \hat f(k) \overline{\hat g(k)}\\
-c_{\kappa,L}\sum_{\sigma \in \mathfrak A} \frac1{L^5}\sum_{\vec k \in C_{\sigma,\bar \kappa}} 
\frac1{\Delta(\vec k)} \prod_{j=1}^5 |a(k_j)|^2 \hat f(k) \overline{\hat g(k)}.
\end{multline}
For the same reasons as in the proof of Proposition \ref{prop:sumtointegral}, we get that
\begin{multline*}
\sum_{\sigma \in \mathfrak A} \frac1{L^5}\sum_{\vec k \in C_{\sigma,\kappa}} 
\frac{ \cos(\Delta(\vec k) t\rho) -1}{\Delta(\vec k)} \prod_{j=1}^5 |a(k_j)|^2 \hat f(k) \overline{\hat g(k)} = \\
\frac{12}{3^5} \int_{\vec k \in B}  \frac{ \cos(\Delta(\vec k) t\rho) -1}{\Delta(\vec k)} \prod_{j=1}^5 |a(k_j)|^2 \hat f(k) \overline{\hat g(k)} + O_{a,f,g}\Big( \frac{\rho^2}{L} + \frac{\rho\mu}\nu + \frac\rho\mu\Big)
\end{multline*}
the key points being that $\frac{ \cos(\Delta(\vec k) t\rho) -1}{\Delta(\vec k)} \leq t\rho$ and that the derivative of $\frac{\cos x -1 }{x}$ is continuous and bounded. 

This erases the first two lines in \eqref{eqKKappa} as in
\begin{multline}\label{eqKKappa2}
K_{\kappa,L} + K_{\bar \kappa,L} = 
 c_{\kappa,L}\sum_{\sigma \in \mathfrak A} \frac1{L^5}\sum_{\vec k \in C_{\sigma,\kappa}} 
\frac1{\Delta(\vec k)} \prod_{j=1}^5 |a(k_j)|^2 \hat f(k) \overline{\hat g(k)}\\
-c_{\kappa,L}\sum_{\sigma \in \mathfrak A} \frac1{L^5}\sum_{\vec k \in C_{\sigma,\bar \kappa}} 
\frac1{\Delta(\vec k)} \prod_{j=1}^5 |a(k_j)|^2 \hat f(k) \overline{\hat g(k)}\\
+ O_{a,f,g}\Big( \frac{\rho^2}{L} + \frac{\rho\mu}\nu + \frac\rho\mu\Big)
\end{multline}

We deal with the remainder, we set
\[
K_{\kappa,L,2} = 
\frac1{L^5}\sum_{\vec k \in C_{\sigma,\kappa}} 
\frac1{\Delta(\vec k)} \prod_{j=1}^5 |a(k_j)|^2 \hat f(k) \overline{\hat g(k)}
-\frac1{L^5}\sum_{\vec k \in C_{\sigma,\bar \kappa}} 
\frac1{\Delta(\vec k)} \prod_{j=1}^5 |a(k_j)|^2 \hat f(k) \overline{\hat g(k)}
\]
Writing $C_\kappa = \{\vec k | Lk \in \kappa(0) + 3\Z, \;\forall j ,\; Lk_j \in \kappa(j) + 3\Z\}$ and $\vec j = \vec k + \frac{\bar \kappa - \kappa}{L}$, we get
\begin{multline*}
K_{\kappa,L,2} = \frac1{L^5} \sum_{\vec k \in C_\kappa} \Big[\frac{{\bf 1}_{\vec k \in C_\sigma} - {\bf 1}_{\vec j \in C_\sigma}}{\Delta(\vec k)} + {\bf 1}_{\vec j\in C_\sigma}\Big( 
\frac1{\Delta(\vec k)} - \frac1{\Delta(\vec j)}\Big)\Big] \prod_{j=1}^5 |a(k_j)|^2 \hat f(k) \overline{\hat g(k)}\\
+ \frac1{L^5} \sum_{\vec k \in C_\kappa} \frac{{\bf 1}_{\vec j\in C_\sigma}}{\Delta(\vec j)}\Big(  \prod_{j=1}^5 |a(k_j)|^2 \hat f(k) \overline{\hat g(k)} -  \prod_{j=1}^5 |a(j_j)|^2 \hat f(j) \overline{\hat g(j)}\Big)
\end{multline*}

We divide again $K_{\kappa,L,2}$ into three parts as
\[
K_{\kappa,L,2} = K_{\kappa,L,3} + K_{\kappa,L,4} + K_{\kappa,L,5}
\]
with
\[
K_{\kappa,L,3} =  \frac1{L^5} \sum_{\vec k \in C_\kappa} \frac{{\bf 1}_{\vec k \in C_\sigma} - {\bf 1}_{\vec j \in C_\sigma}}{\Delta(\vec k)} \prod_{j=1}^5 |a(k_j)|^2 \hat f(k) \overline{\hat g(k)},
\]

\[
K_{\kappa,L,4} =  \frac1{L^5} \sum_{\vec k \in C_\kappa}  {\bf 1}_{\vec j\in C_\sigma}\Big( 
\frac1{\Delta(\vec k)} - \frac1{\Delta(\vec j)}\Big) \prod_{j=1}^5 |a(k_j)|^2 \hat f(k) \overline{\hat g(k)}
\]
and
\[
K_{\kappa,L,5} = \frac1{L^5} \sum_{\vec k \in C_\kappa} \frac{{\bf 1}_{\vec j\in C_\sigma}}{\Delta(\vec j)}\Big(  \prod_{j=1}^5 |a(k_j)|^2 \hat f(k) \overline{\hat g(k)} -  \prod_{j=1}^5 |a(j_j)|^2 \hat f(j) \overline{\hat g(j)}\Big).
\]

We first deal with $K_{\kappa,L,5}$. We have $|\vec k - \vec j|\leq \frac2{L}$. Hence,
\[
\Big|  \frac{{\bf 1}_{\vec j\in C_\sigma}}{\Delta(\vec j)}\Big(  \prod_{j=1}^5 |a(k_j)|^2 \hat f(k) \overline{\hat g(k)} -  \prod_{j=1}^5 |a(j_j)|^2 \hat f(j) \overline{\hat g(j)}\Big)\Big|\lesssim_{a,f,g} \frac{{\bf 1}_{\vec j\in C_\sigma}\nu}{L}.
\]

Besides $a$ has compact support hence there exists $M>0$ such that
\[
|K_{\kappa,L,5}| \lesssim_{a,f,g}  \frac1{L^5} \sum_{\vec k \in C_\kappa\cap [-M,M]^6}\frac{{\bf 1}_{\vec j\in C_\sigma}\nu}{L} \lesssim_{a,f,g} \frac{\nu}{L}.
\]

We turn to $K_{\kappa,L,4}$. Given again the fact that $a$ has compact support, we get on the support of $\prod_{j=1}^5 |a(k_j)|^2$ remembering that $k = -\sum_{j=1}^5 (-1)^j k_j$,
\[
\Big|{\bf 1}_{\vec j\in C_\sigma}\Big( 
\frac1{\Delta(\vec k)} - \frac1{\Delta(\vec j)}\Big)\Big| \lesssim_a \frac{{\bf 1}_{\vec j\in C_\sigma} \nu}{|\Delta(\vec k)|L}.
\]
Since $|\Delta(\vec k) | \geq |\Delta(\vec j)| - |\Delta(\vec j) - \Delta(\vec k)|$ and since
$|\Delta(\vec j) - \Delta(\vec k)|\lesssim_a L^{-1}$ we get that above a certain rank, since $\nu = o(\sqrt L)$,
\[
\Big|{\bf 1}_{\vec j\in C_\sigma}\Big( 
\frac1{\Delta(\vec k)} - \frac1{\Delta(\vec j)}\Big)\Big| \lesssim_a \frac{{\bf 1}_{\vec j\in C_\sigma} \nu^2}{L}.
\]
And therefore, we get
\[
|K_{\kappa,L,4}| \lesssim_{a,f,g} \frac{\nu^2}{L}.
\]

We now turn to $K_{\kappa,L,3}$ and estimate the numbers of $\vec k$ such that
\[
{\bf 1}_{\vec k \in C_\sigma} - {\bf 1}_{\vec j \in C_\sigma}
\]
is not null. We assume without loss of generality that $\vec k \in C_\sigma$ but that $\vec j\notin C_\sigma$.

\textbf{First Case : } $|\Delta(\vec j)| < \nu^{-1}$. We have 
\[
\nu^{-1} \leq |\Delta(\vec k)| \leq |\Delta (\vec j)| + |\Delta(\vec k) - \Delta(\vec j)|.
\]
Since on the support of  $\prod_{j=1}^5 |a(k_j)|^2$, we have 
\[
|\Delta(\vec k) - \Delta(\vec j)| \lesssim_a L^{-1},
\]
we get
\[
\nu^{-1} \leq |\Delta(\vec k)| \leq |\Delta (\vec j)| +C_a L^{-1}.
\]
We recall that $\Delta(\vec k) = \bar D(k,k_1,k_2,k_3) -  2\bar k k_4$ and thus 
\[
\frac{1}{\nu |\bar k|} \leq |k_4 - \frac{\bar D}{2\bar k}|\leq \frac{1}{\nu |\bar k|} + C_a \frac\mu\nu.
\]
Hence $k_4$ belongs to the reunion of two intervals of size $C_a\frac\mu{L}$. 

\textbf{Second case : } $|\bar j| < \mu^{-1}$. We have 
\[
\mu^{-1} \leq |\bar k| \leq |\bar j| + |\bar j -\bar k| \leq \mu^{-1} + \frac2{L}.
\]
Therefore, $k_3$ belongs to the reunion of two intervals of size $\frac2{L}$.

\textbf{Third case : } $|\bar j_\sigma |< \mu^{-1}$. Similar to second case. 

Therefore,
\[
|K_{\kappa,L,3}|\lesssim_{a,f,g} \frac{\mu\nu}{L} + \frac{\nu}{L}
\]
and we can conclude.

\end{proof}

From Proposition \ref{prop:diracdeltas}, we therefore get
\[
\lim_{L\rightarrow \infty} I_L(t) = \frac1{12\pi^4}\int_{\R^6} \delta(k + \sum_{j=1}^5 (-1)^j k_j)  \delta (\Delta (\vec k) ) \frac1{k-k_1 + k_2 - k_3 } \prod_{j=1}^5|a(k_j)|^2\hat f(k) \overline{\hat g(k)} d\vec k
\]
and get Theorem \ref{theo:main2}.

\appendix

\section{Tree glossary}\label{app:gloss}

\subsection{Labeled trees}

We draw some trees corresponding to Definition \ref{def:labelledTrees} in Figures \ref{fig:labelledtrees01} and \ref{fig:labelledtrees2}. The squares represent leaves and the circles represent nodes. We have $T_0 \in \mathcal T_0[k]$, $T_1 \in \mathcal T_1[k]$ and $T_2 \in \mathcal T_2[k]$. We fix $\Delta = \Delta_1 = k^2 - k_1^2 + k_2^2 - k_3^2 + k_4^2 - k_5^2$, and $\Delta_2 = k_1^2 - j_1^2 + j_2^2 - j_3^2 + j_4^2-j_5^2$. Note that, in mathematical writing, we have 
\[
T_0 = (k),\quad T_1 = ((k_1),(k_2),(k_3),(k_4),(k_5),k), 
\]
and
\[ 
T_2 = (((j_1),(j_2),(j_3),(j_4),(j_5),k_1),(k_2),(k_3),(k_4),(k_5),k).
\]

\begin{figure}[!h]\begin{center}
\includegraphics[width=10cm]{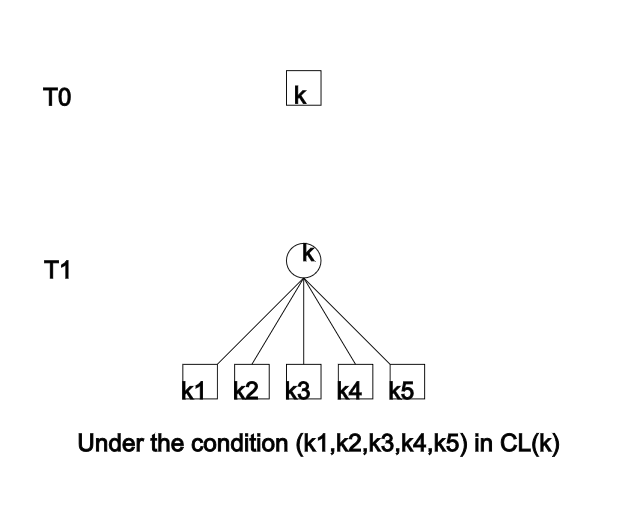}
\caption{Labelled trees with 0 and 1 nodes}
\label{fig:labelledtrees01}
\end{center}
\end{figure}

\begin{figure}[!h]\begin{center}
\includegraphics[width=10cm]{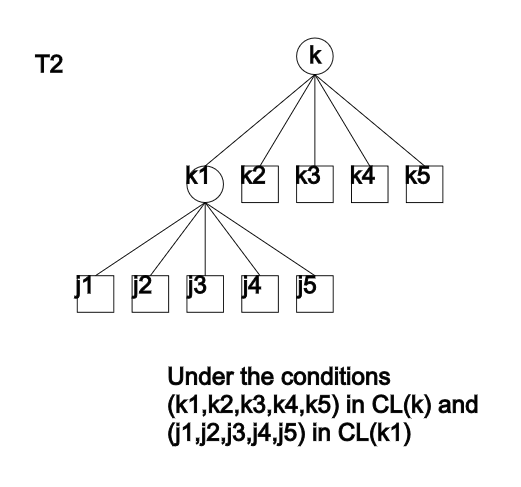}
\caption{Labelled tree with 2 nodes}
\label{fig:labelledtrees2}
\end{center}
\end{figure}

Keeping in mind the previous examples, we have corresponding to Definition \ref{def:functionsTrees} :
\begin{eqnarray*}
F_{T_0}(t)&=& 1 ,\\
F_{T_1} (t) &=& -i\int_{0}^t e^{i\Delta \tau}d\tau ,\\
F_{T_2}(t) &=& -\int_{0}^t e^{i\Delta_1 \tau}\int_{0}^\tau e^{i\Delta_2 s }ds d\tau ,\\
g_{T_0} &=& g_{Lk},  \\
 g_{T_1} &=& g_{Lk_1}\diamond \bar g_{Lk_2} \diamond g_{Lk_3} \diamond \bar g_{Lk_4} \diamond g_{Lk_5}, \\
 g_{T_2} &=&   g_{Lj_1}\diamond \bar g_{Lj_2} \diamond g_{Lj_3} \diamond \bar g_{Lj_4}\diamond g_{Lj_5}\diamond \bar g_{Lk_2} \diamond g_{Lk_3} \diamond \bar g_{Lk_4}\diamond g_{Lk_5}, 
\end{eqnarray*}
and 
\begin{eqnarray*}
A_{T_0} &=& a(k), \\
A_{T_1} &=& a(k_1)\bar a(k_2) a(k_3) \bar a(k_4) a(k_5), \\
A_{T_1} &=&  a(j_1)\bar a(j_2) a(j_3) \bar a(j_4) a(j_5)\bar a(k_2) a(k_3) \bar a(k_4) a(k_5).
\end{eqnarray*}

We have, corresponding to Definition \ref{def:labelsTrees},
\[
\vec k_0:=\vec T_0 = k,\quad \vec k_1:= \vec T_1 = (k_1,k_2,k_3,k_4,k_5),\quad \vec k_2 := \vec T_2 = (j_1,j_2,j_3,j_4,j_5,k_2,k_3,k_4,k_5).
\]

\newpage

\subsection{Unlabeled trees}

We draw a picture corresponding to Definition \ref{def:unlabelledtree} in Figure \ref{fig:unlabelledtrees}. In mathematical writing, we have $\texttt{t}_0 = \bot$, $\texttt{t}_1 = (\bot,\bot,\bot,\bot,\bot)$ and $\texttt{t}_2 = ((\bot,\bot,\bot,\bot,\bot),\bot,\bot,\bot,\bot)$. We also have $T_0 = \texttt{t}_0(k)$, $T_1 = \texttt{t}_1((k_1,k_2,k_3,k_4,k_5))$ and 
\[
T_2 = \texttt{t}_2((j_1,j_2,j_3,j_4,j_5,k_2,k_3,k_4,k_5)).
\]

\begin{figure}[!h]\begin{center}
\includegraphics[width=10cm]{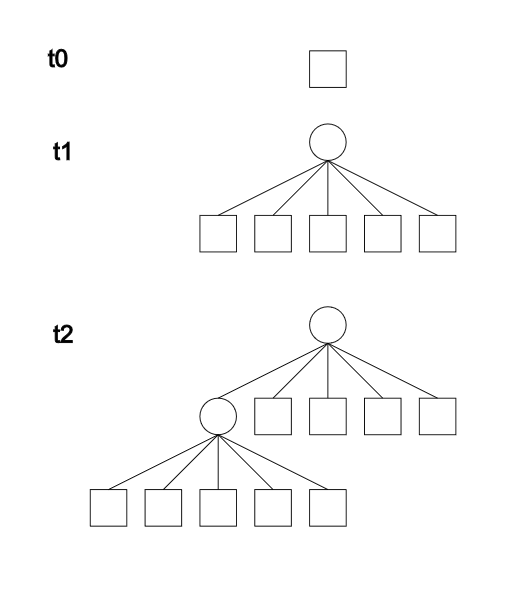}
\caption{Unlabelled trees with 0, 1 and 2 nodes}
\label{fig:unlabelledtrees}
\end{center}
\end{figure}

We draw a picture corresponding to Definition \ref{def:SetOrder} in Figure \ref{fig:order}. We forgot some parenthesis as they were redundant.

\begin{figure}[!h]\begin{center}
\includegraphics[width=15cm]{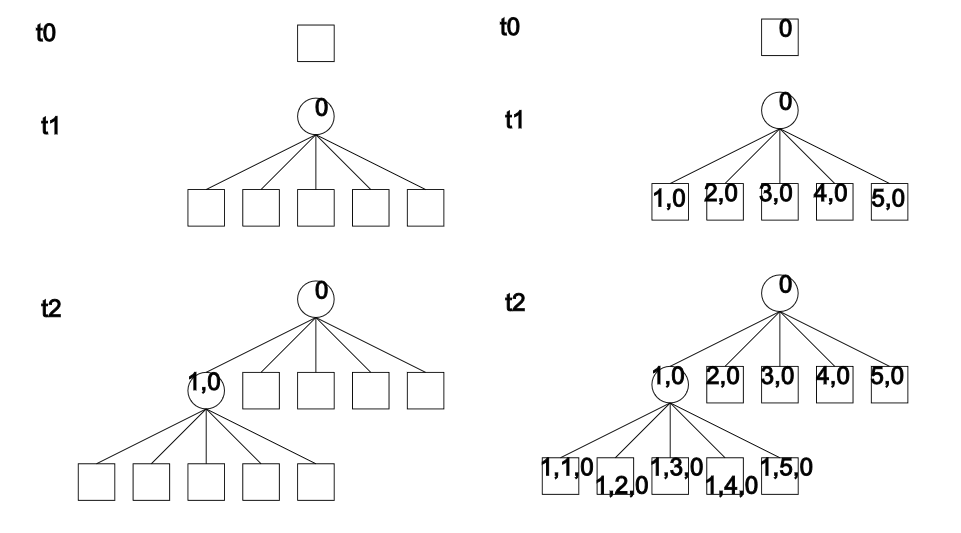}
\caption{Ordering labels of nodes and leaves}
\label{fig:order}
\end{center}
\end{figure}

Corresponding to Definition \ref{def:functionLabels}, we have 
\[
k_{\texttt t_0, \vec k_0}  : 0 \mapsto k, \quad
k_{\texttt t_1,\vec k_1} : \begin{array}{c} 0 \mapsto k \\ (l,0) \mapsto k_l\end{array}
\]
with $k = \sum_{l=1}^5 (-1)^{l+1} k_l$ and 
\[
k_{\texttt t_2,\vec k_2} : \begin{array}{c}
0 \mapsto k \\
(l,0) \mapsto k_l \\
(1,l,0) \mapsto j_l \end{array}
\]
with $k_1 = \sum_{l=1}^5 (-1)^{l+1} j_l$ and $k = \sum_{l=1}^5 (-1)^{l+1}k_l$.

We also have 
\[
\Omega_{\texttt t_1,\vec k_1} : 0 \mapsto \Delta,\quad \Omega_{\texttt t_2,\vec k_2} : \begin{array}{c} 0 \mapsto \Delta_1 \\ (1,0) \mapsto \Delta_2 \end{array} .
\]

\newpage

\subsection{Node ordering}

Finally, corresponding to Definition \ref{def:orderonnodes}, we write explicitly the partial order $R_{\texttt t_4}$ on the tree with four nodes of Figure \ref{fig:treeWithFourNodes}. We have 
\[
(1,4,0) R_{\texttt t_4} (1,0) R_{\texttt t_4} 0 \quad \textrm{and} \quad (5,0) R_{\texttt t_4} 0
\]
and the other nodes are not comparable. In other words, $(1,0)$ is not comparable to $(5,0)$ but also $(1,4,0)$ is not comparable to $(5,0)$. 

\begin{figure}[!h]\begin{center}
\includegraphics[width=10cm]{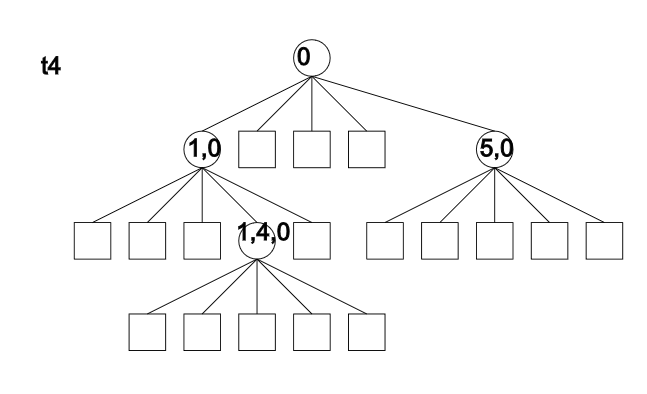}
\caption{A tree with four nodes}
\label{fig:treeWithFourNodes}
\end{center}
\end{figure}

\newpage

\section{A program in Python}\label{app:program}

Here, we present a program on Python designed to give the number of maps $\kappa:[0,4] \rightarrow \{-1,0,1\}$ such that the number
\[
\kappa(0)^2 + \sum_{j=1}^5 (-1)^j \kappa(j)^2
\]
is equal to $0$, $1$ or $-1$ in $\mathbb F_3$. 

\noindent\texttt{
def hs(k) :\\
\indent    a,b,c = 0,0,0\\
\indent    for x in range(3):\\
\indent\indent         for x1 in range(3):\\
\indent\indent\indent            for x2 in range(3):\\
\indent\indent\indent\indent                for x3 in range(3):\\
\indent\indent\indent\indent\indent                    for x4 in range(3):\\
\indent\indent\indent\indent\indent\indent                        if (x**2-x1**2+x2**2-x3**2+x4**2-(x-x1+x2-x3+x4)**2)\%3 == 0: a=a+1\\
\indent\indent\indent\indent\indent\indent                        else:\\
\indent\indent\indent\indent\indent\indent\indent                            if (x**2-x1**2+x2**2-x3**2+x4**2-(x-x1+x2-x3+x4)**2)\%3 == 1: b=b+1\\
\indent\indent\indent\indent\indent\indent\indent                            else: c=c+1\\
\indent    return a,b,c}
    
The program returns : $(99,72,72)$. 

\bibliographystyle{amsplain}
\bibliography{SQbib} 

\end{document}